\documentclass[a4paper]{amsproc}
\pdfoutput=1

\usepackage{amssymb}
\usepackage{multirow}
\usepackage{rotating}
\usepackage{arydshln}
\usepackage{enumitem}
\usepackage{natbib}
\usepackage{calc}
\usepackage[T1]{fontenc}   % for bold \Singular
\usepackage[hang]{caption}
\usepackage[hyperfootnotes=false]{hyperref}
\usepackage{tikz}
\usetikzlibrary{matrix}

\newtheorem{theorem}{Theorem}
\newtheorem{defn}[theorem]{Definition}
\newtheorem{lemma}[theorem]{Lemma}
\theoremstyle{definition}
\newtheorem{remark}[theorem]{Remark}
\newtheorem{example}[theorem]{Example}

\let\originalleft\left
\let\originalright\right
\renewcommand{\left}{\mathopen{}\mathclose\bgroup\originalleft}
\renewcommand{\right}{\aftergroup\egroup\originalright}

\newcommand{\myunderbrace}[3]{%
\makebox[\minof{0pt}{\widthof{$#1$}/2-\widthof{$#2$}/2-\widthof{$#3$}}]{}%
\underbrace{#1}_{\displaystyle\phantom{#3}#2 #3}}

\newcommand{\Singular}{\textsc{Singular}}
\newcommand{\realclassify}{\texttt{realclassify.lib}}
\newcommand{\tY}{\widetilde{Y}}

\DeclareMathOperator{\jet}{jet}
\DeclareMathOperator{\N}{\mathbb{N}}
\DeclareMathOperator{\Q}{\mathbb{Q}}
\DeclareMathOperator{\R}{\mathbb{R}}
\DeclareMathOperator{\C}{\mathbb{C}}
\DeclareMathOperator{\K}{\mathbb{K}}
\DeclareMathOperator{\A}{\mathbb{A}}
\DeclareMathOperator{\NF}{NF}
\DeclareMathOperator{\s}{S}
\DeclareMathOperator{\T}{T}
\DeclareMathOperator{\Aut}{Aut}
\DeclareMathOperator{\id}{id}
\DeclareMathOperator{\Quot}{Quot}
\DeclareMathOperator{\Imag}{Im}
\DeclareMathOperator{\dash}{\!\textnormal{-}\!}
\newcommand{\csim}{\ensuremath{\mathrel{%
  \raisebox{0pt}[\heightof{$\overset{\scriptscriptstyle\C}{\sim}$}]{%
    $\overset{\C}{\sim}$%
  }%
}}}
\newcommand{\ksim}{\ensuremath{\mathrel{%
  \raisebox{0pt}[\heightof{$\overset{\scriptscriptstyle\K}{\sim}$}]{%
    $\overset{\K}{\sim}$%
  }%
}}}

\title[The Classification of Real Singularities Using \textsc{Singular}, %
Part II]%
{The Classification of Real Singularities Using \textsc{Singular}\\
Part II: The Structure of the Equivalence Classes of the Unimodal %
Singularities}

\author{Magdaleen S.\@ Marais}
\address{Magdaleen S.\@ Marais\\
University of Pretoria and African Institute for Mathematical Sciences\\
Department of Mathematics and Applied Mathematics\\
Private bag X20\\
Hatfield 0028\\
South Africa}
\email{magdaleen.marais@up.ac.za}

\author{Andreas Steenpa\ss}
\address{Andreas Steenpa\ss\\
Department of Mathematics\\
University of Kaiserslautern\\
Erwin-Schr\"odinger-Str.\\
67663 Kaiserslautern\\
Germany}
\email{steenpass@mathematik.uni-kl.de}

\thanks{This research project was supported by the African Institute for
Mathematical Sciences, the German National Academic Foundation, and grants
awarded by Wolfram Decker and Gert-Martin Greuel. We are thankful to all of
them.}


\begin{document}

\begin{abstract}
In the classification of real singularities by \citet{AVG1985}, normal forms,
as representatives of equivalence classes under right equivalence, are not
always uniquely determined. We describe the complete structure of the
equivalence classes of the unimodal real singularities of corank~$2$. In other
words, we explicitly answer the question which normal forms of different type
are equivalent, and how a normal form can be transformed within the same
equivalence class by changing the value of the parameter. This provides new
theoretical insights into these singularities and has important consequences
for their algorithmic classification.
\end{abstract}

\maketitle

\section{Introduction}

This article is the second part of a series of articles on the algorithmic
classification of real singularities up to modality $1$ and corank $2$. The
first part \citep{MS2015} covers the splitting lemma and the simple
singularities. All the algorithms presented there have been implemented in the
computer algebra system \Singular{} \citep{DGPS} as a library called
\realclassify{} \citep{realclassify}.

Both real and complex singularities have been extensively studied
\citep[see][]{Siersma, A1976, BG1982, GH1993, FK1999, GLS2007, FKN2010}. Our
work is based on the classifications of complex and real singularities of small
modality up to stable equivalence by \citet{AVG1985}. Two power series
$f, g \in \K[[x_1,\ldots,x_n]]$ with a critical point at the origin and
critical value $0$ are complex (if $\K = \C$) or real (if $\K = \R$)
equivalent, denoted by $f \ksim g$, if there exists a $\K$-algebra automorphism
$\phi$ of $\K[[x_1,\ldots,x_n]]$ such that $\phi(f) = g$. They are stable
(complex or real) equivalent if they become (complex or real) equivalent after
the direct addition of non-degenerate quadratic terms.

\begin{table}[p]
\centering
\caption{Normal forms of singularities of modality~$1$ and corank~$2$ as given
in \citet{AVG1985}}
\label{tab:normal_forms}
\begin{minipage}{\textwidth}
\renewcommand{\thempfootnote}{\fnsymbol{mpfootnote}}
\addtocounter{mpfootnote}{1}
\newcommand{\setfnA}{\footnote{\label{fnA}%
Note that the restriction $a^2 \neq 4$ applies to the normal forms of the real
subtypes $X_9^{++}$, $X_9^{--}$, and $J_{10}^+$ as well as to the normal forms
of the complex types $X_9$ and $J_{10}$ while the restriction $a^2 \neq -4$
applies to the normal forms of the real subtypes $X_9^{+-}$, $X_9^{-+}$, and
$J_{10}^-$ if we allow complex parameters.}}
\newcommand{\reffnA}{\textsuperscript{\ref*{fnA}}}
\centering
\begin{tabular}{|c|c|c|c|c|}
\hline

\multicolumn{1}{|c}{}
 & & Complex     & Normal forms     & \multirow{2}{*}{Restrictions} \\[-0.5ex]
\multicolumn{1}{|c}{}
 & & normal form & of real subtypes &                               \\
\hline\hline

\multirow{6}{*}{\begin{sideways}Parabolic\end{sideways}}

& \multirow{4}{*}{$X_9$} & \multirow{4}{*}{$x^4+ax^2y^2+y^4$}
  & $+x^4+ax^2y^2+y^4$ $(X_9^{++})$ & \multirow{2}{*}{$a^2\neq+4$\setfnA}
\\ \cline{4-4}
&&& $-x^4+ax^2y^2-y^4$ $(X_9^{--})$ &
\\ \cline{4-5}
&&& $+x^4+ax^2y^2-y^4$ $(X_9^{+-})$ & \multirow{2}{*}{$a^2\neq-4$\reffnA}
\\ \cline{4-4}
&&& $-x^4+ax^2y^2+y^4$ $(X_9^{-+})$ &
\\ \cline{2-5}

& \multirow{2}{*}{$J_{10}$} & \multirow{2}{*}{$x^3+ax^2y^2+xy^4$}
  & $x^3+ax^2y^2+xy^4$ $(J_{10}^+)$ & $a^2 \neq +4$\reffnA \\ \cline{4-5}
&&& $x^3+ax^2y^2-xy^4$ $(J_{10}^-)$ & $a^2 \neq -4$\reffnA \\ \hline

\multirow{12}{*}{\begin{sideways}Hyperbolic\end{sideways}}

& \multirow{2}{*}{$J_{10+k}$} & \multirow{2}{*}{$x^3+x^2y^2+ay^{6+k}$}
  & $x^3+x^2y^2+ay^{6+k}$ $(J_{10+k}^+)$
      & \multirow{2}{*}{$a \neq 0,\; k > 0$} \\ \cline{4-4}
&&& $x^3-x^2y^2+ay^{6+k}$ $(J_{10+k}^-)$ &   \\ \cline{2-5}

& \multirow{4}{*}{$X_{9+k}$} & \multirow{4}{*}{$x^4+x^2y^2+ay^{4+k}$}
  & $+x^4+x^2y^2+ay^{4+k}$ $(X_{9+k}^{++})$
      & \multirow{4}{*}{$a \neq 0,\; k > 0$}  \\ \cline{4-4}
&&& $-x^4-x^2y^2+ay^{4+k}$ $(X_{9+k}^{--})$ & \\ \cline{4-4}
&&& $+x^4-x^2y^2+ay^{4+k}$ $(X_{9+k}^{+-})$ & \\ \cline{4-4}
&&& $-x^4+x^2y^2+ay^{4+k}$ $(X_{9+k}^{-+})$ & \\ \cline{2-5}

& \multirow{4}{*}{$Y_{r,s}$} & \multirow{4}{*}{$x^2y^2+x^r+ay^s$}
  & $+x^2y^2+x^r+ay^s$ $(Y_{r,s}^{++})$
      & \multirow{4}{*}{$a \neq 0,\; r,s > 4$} \\ \cline{4-4}
&&& $-x^2y^2-x^r+ay^s$ $(Y_{r,s}^{--})$ &      \\ \cline{4-4}
&&& $+x^2y^2-x^r+ay^s$ $(Y_{r,s}^{+-})$ &      \\ \cline{4-4}
&&& $-x^2y^2+x^r+ay^s$ $(Y_{r,s}^{-+})$ &      \\ \cdashline{2-3}\cline{4-5}

& \multirow{2}{*}{$\tY_r$} & \multirow{2}{*}{$(x^2+y^2)^2+ax^r$}
  & $+(x^2+y^2)^2+ax^r$ $(\tY_r^+)$
      & \multirow{2}{*}{$a \neq 0,\; r > 4$} \\ \cline{4-4}
&&& $-(x^2+y^2)^2+ax^r$ $(\tY_r^-)$ &        \\ \hline

\multirow{12}{*}{\begin{sideways}Exceptional\end{sideways}}

& $E_{12}$ & $x^3+y^7+axy^5$ & $x^3+y^7+axy^5$ & - \\ \cline{2-5}

& $E_{13}$ & $x^3+xy^5+ay^8$ & $x^3+xy^5+ay^8$ & - \\ \cline{2-5}

& \multirow{2}{*}{$E_{14}$} & \multirow{2}{*}{$x^3+y^8+axy^6$}
  & $x^3+y^8+axy^6$ $(E_{14}^+)$ & \multirow{2}{*}{-} \\ \cline{4-4}
&&& $x^3-y^8+axy^6$ $(E_{14}^-)$ &                    \\ \cline{2-5}

& $Z_{11}$ & $x^3y+y^5+axy^4$ & $x^3y+y^5+axy^4$ & - \\ \cline{2-5}

& $Z_{12}$ & $x^3y+xy^4+ax^2y^3$ & $x^3y+xy^4+ax^2y^3$ & - \\ \cline{2-5}

& \multirow{2}{*}{$Z_{13}$} & \multirow{2}{*}{$x^3y+y^6+axy^5$}
  & $x^3y+y^6+axy^5$ $(Z_{13}^+)$ & \multirow{2}{*}{-} \\ \cline{4-4}
&&& $x^3y-y^6+axy^5$ $(Z_{13}^-)$ &                    \\ \cline{2-5}

& \multirow{2}{*}{$W_{12}$} & \multirow{2}{*}{$x^4+y^5+ax^2y^3$}
  & $+x^4+y^5+ax^2y^3$ $(W_{12}^+)$ & \multirow{2}{*}{-} \\ \cline{4-4}
&&& $-x^4+y^5+ax^2y^3$ $(W_{12}^-)$ &                    \\ \cline{2-5}

& \multirow{2}{*}{$W_{13}$} & \multirow{2}{*}{$x^4+xy^4+ay^6$}
  & $+x^4+xy^4+ay^6$ $(W_{13}^+)$ & \multirow{2}{*}{-} \\ \cline{4-4}
&&& $-x^4+xy^4+ay^6$ $(W_{13}^-)$ &                    \\ \hline

\end{tabular}
\end{minipage}
\end{table}

In this article, we focus on the unimodal singularities of corank 2. Their
complex and real normal forms can be found in Table~\ref{tab:normal_forms}.
Just as for the simple singularities \citep[cf.][]{MS2015}, it turns out that
the complex singularity types split up into one or several real subtypes and
that the normal forms of the real subtypes belonging to the same complex type
differ from each other only in the signs of some terms. We therefore sometimes
refer to the complex singularity types as main types. The hyperbolic type
$\tY_r$ is an exception because it is complex equivalent to $Y_{r,r}$ and only
occurs as a type on its own in the real classification.

Some of the normal forms in Table~\ref{tab:normal_forms} are equivalent to
others. Such equivalences occur both between different real subtypes and
between normal forms with different values of the parameter $a$. However, there
are no equivalences between different main types. To give an example,
$x^4-4x^2y^2+y^4$, the normal form of $X_9^{++}$ with $a = -4$, is equivalent
to $-x^4+10x^2y^2-y^4$, the normal form of $X_9^{--}$ with $a = 10$, via the
coordinate transformation $x \mapsto c(x+y)$, $y \mapsto c(x-y)$ with
$c = \frac{1}{\sqrt[4]{2}}$. Examples like this one have consequences for the
algorithmic classification of real singularities. The question if the
singularity in the example is of real type $X_9^{++}$ or of real type
$X_9^{--}$ is not well-posed and the value of the parameter is not uniquely
determined. Note that this problem does not occur for the simple singularities:
By definition, their normal forms do not admit parameters, and there are no
equivalences between different real subtypes except for the main types $A_k$
where $k$ is even, cf.\@ \citet{MS2015}.

The goal of this article is to determine the complete structure of the
equivalence classes for the unimodal real singularities of corank $2$. Based on
these results, we will present algorithms to determine the equivalence class of
a given unimodal real singularity of corank $2$ in a subsequent part of this
series of articles. If $T_1$ and $T_2$ are subtypes of the same singularity
main type $T$ and if $g_1(a)$ and $g_2(a)$ are the normal forms of $T_1$ and
$T_2$, respectively, where $a$ denotes the value of the parameter, then we are
interested in the set of all pairs $(u, v)$ such that $g_1(u)$ is equivalent to
$g_2(v)$. This question can be asked in three different ways: If we consider
complex values of $u$ and $v$ and complex coordinate transformations, we denote
the corresponding set by $P_1(T_1, T_2)$, for real values of $u$ and $v$, but
still complex transformations by $P_2(T_1, T_2)$, and finally by
$P_3(T_1, T_2)$ if we consider only real values of $u$ and $v$ and real
transformations, cf.\@ Definition~\ref{def:Psets}. We study all three sets
$P_1$, $P_2$, and $P_3$ for two reasons: First, we need $P_1$ and $P_2$ to
compute $P_3$ which is the set we are mainly interested in for our purpose.
Second, besides the real structure, we thus also obtain the complete complex
structure of the equivalence classes for all the specific normal forms given in
Table~\ref{tab:normal_forms} which is an interesting result on its own.

The formal definitions of these sets and other basic notations are introduced
in Section~\ref{sec:Psets}, along with different ways how $P_1(T_1, T_2)$,
$P_2(T_1, T_2)$, and $P_3(T_1, T_2)$ can be conveniently written down in
concrete cases. The following sections are devoted to the computation of these
sets for any two real subtypes $T_1$ and $T_2$ listed in
Table~\ref{tab:normal_forms}. We first recall the definitions of (piecewise)
weighted jets and filtrations in Section~\ref{sec:weights}. They play a major
role in the proof of Theorem~\ref{thm:sufficient_sets}, the main result of
Section~\ref{sec:sufficient_sets}. This theorem allows us to restrict ourselves
to a small subset of coordinate transformations, which we call a sufficient
set, if we want to determine $P_1(T_1, T_2)$. It is thus the theoretic basis
for Section~\ref{sec:computations} where we explain how $P_1(T_1, T_2)$,
$P_2(T_1, T_2)$, and $P_3(T_1, T_2)$ can be computed using \Singular{}. We also
give an example with explicit \Singular{} commands. These methods do not apply
for the singularity type $\tY_r$ which is treated separately in
Section~\ref{ssec:Yr}. Section~\ref{sec:results} contains the results of these
computations in a concise form. Finally, we point out some remarkable aspects
of the results in this article as well as their consequences for the
algorithmic classification of the unimodal real singularities of corank $2$ in
Section~\ref{sec:interpretation}. The maybe most surprising outcome is that the
real subtype $J_{10}^-$ is actually redundant whereas $J_{10}^+$ is not.

\section{The Sets of Parameter Transformations
\texorpdfstring{$\boldsymbol{P_1}$, $\boldsymbol{P_2}$, and
$\boldsymbol{P_3}$}{P1, P2, and P3}}%
\label{sec:Psets}

Let us start with some basic definitions. Throughout the rest of this article,
let $\K$ be, in each case, either $\R$ or $\C$.

\begin{defn}
Two power series $f, g \in \K[[x_1,\ldots,x_n]]$ are called $\K$-equivalent,
denoted by $f \ksim g$, if there exists a $\K$-algebra automorphism
$\phi$ of $\K[[x_1,\ldots,x_n]]$ such that $\phi(f) = g$.
\end{defn}

Note that $\ksim$ is an equivalence relation on $ \K[[x_1,\ldots,x_n]]$.
\citet{AVG1985} give the following formal definition for normal forms w.r.t.\@
this relation:

\begin{defn}
Let $K \subset \K[[x_1,\ldots,x_n]]$ be a union of equivalence classes w.r.t.\@
the relation $\ksim$. A \emph{normal form} for $K$ is given by a smooth map
\[
\Phi: B \longrightarrow \K[x_1,\ldots,x_n] \subset \K[[x_1,\ldots,x_n]]
\]
of a finite-dimensional $\K$-linear \emph{space of parameters} $B$ into the
space of polynomials for which the following three conditions hold:

\begin{enumerate}
\item $\Phi(B)$ intersects all the equivalence classes of K;

\item the inverse image in $B$ of each equivalence class is finite;

\item the inverse image of the whole complement to $K$ is contained in some
proper hypersurface in $B$.
\end{enumerate}
\end{defn}

\begin{remark}
Note that the term \emph{normal form} is subtly ambiguous. According to the
above definition, a normal form is a smooth map where the inverse image of each
equivalence class may contain more than one element, whereas the common meaning
of this term rather refers to the polynomials which are the images under this
map. We could be more precise
and avoid this ambiguity by introducing a new term for either of the two
meanings. However, we stay with the common usage of the term \emph{normal form}
in order to prevent confusion.
\end{remark}

\begin{defn}
Let $S \subset \Aut_{\K}(\K[[x_1,\ldots,x_n]])$ be a set of $\K$-algebra
automorphisms of $\K[[x_1, \ldots, x_n]]$ and let
$f, g \in \K[[x_1, \ldots, x_n]]$ be two power series.
\begin{enumerate}
\item We denote the set of all automorphisms in $S$ which take $f$ to $g$ by
$\T_{\K}^S(f,g)$, i.e.\@
\[
\T_{\K}^S(f,g):=\{\phi\in S\mid \phi(f)=g\}\,.
\]

\item If $S=\Aut_{\K}(\K[[x_1,\ldots,x_n]])$, we simply write
$\T_{\K}(f,g)$ for $\T_{\K}^S(f,g)$, i.e.\@
\[
\T_{\K}(f,g)
:= \{\phi \in \Aut_{\K}(\K[[x_1, \ldots, x_n]]) \mid \phi(f) = g \} \,.
\]
\end{enumerate}
\end{defn}

The above definition is the key ingredient for the definition of $P_1$, $P_2$,
and $P_3$. We also need the following notation.

\begin{remark}
As usual, we denote the field of quotients $\Quot(\K[a])$ by $\K(a)$. Let
$f \in \K(a)[[x_1,\ldots,x_n]]$ be a power series over this quotient field.
Then $f$ can be written as $f = \sum_{\nu \in \N^n} c_\nu \boldsymbol{x}^\nu$
with coefficients $c_\nu = \frac{p_\nu}{q_\nu} \in \K(a)$ where
$p_\nu, q_\nu \in \K[a]$ are polynomials of minimal degree with this property
and $q_\nu \neq 0$ for all $\nu \in \N^n$.

If we consider the polynomials $p_\nu, q_\nu$ as polynomial functions
$p_\nu, q_\nu: \; \K \rightarrow \K$, then we may also consider the
coefficients $c_\nu$ as functions
$c_\nu: \; \K \setminus V(q_\nu) \rightarrow \K$ where $V(q_\nu)$ is the set of
points where $q_\nu$ vanishes. Via this correspondence, we finally get power
series
$f(u) := \sum_{\nu \in \N^n} c_\nu(u) \boldsymbol{x}^\nu
\in \K[[x_1,\ldots,x_n]]$ for each value
$u \in \K \setminus \bigcup_{\nu \in \N^n} V(q_\nu)$.
\end{remark}

We use the notation $f(u)$ throughout this paper. Likewise, we add the
value of the parameter which occurs in the normal form as given in
Table~\ref{tab:normal_forms} in parentheses to the name of the singularity
(sub-)type if we want to refer specifically to the corresponding equivalence
class. For instance, we denote by $E_{14}(3)$ the (complex or real)
right-equivalence class of $x^3+y^8+3xy^6$.

For any specific singularity type $T$, we denote by $\NF(T)$ its normal form as
shown in Table~\ref{tab:normal_forms}, i.e.\@ we write
$\NF\left(E_{14}(a)\right) = \NF\left(E_{14}^+(a)\right)$ for the polynomial
$x^3+y^8+axy^6$ and $\NF\left(E_{14}^-(5)\right)$ for $x^3-y^8+5xy^6$.

We can now state the main definition of this section.

\begin{defn}\label{def:Psets}
\leavevmode
\begin{enumerate}
\item
Given power series $f,g \in \C(a)[[x_1,\ldots,x_n]]$, we define the
first set of parameter transformations of $f$ and $g$ as
\begin{align*}
P_1(f, g)
:= \bigl\{ (u, v) \in \C^2 \mid
&f(u) \text{ and } g(v) \text{ are well-defined and } \\
&\T_{\C}(f(u), g(v)) \neq \varnothing \bigr\} \,.
\end{align*}

\item
Given power series $f,g \in \R(a)[[x_1,\ldots,x_n]]$, we define the
second set of parameter transformations of $f$ and $g$ as
\begin{align*}
P_2(f, g)
:= \bigl\{ (u, v) \in \R^2 \mid
&f(u) \text{ and } g(v) \text{ are well-defined and } \\
&\T_{\C}(f(u), g(v)) \neq \varnothing \bigr\} \,.
\end{align*}

\item
Given power series $f,g \in \R(a)[[x_1,\ldots,x_n]]$, we define the
third set of parameter transformations of $f$ and $g$ as
\begin{align*}
P_3(f, g)
:= \bigl\{ (u, v) \in \R^2 \mid
&f(u) \text{ and } g(v) \text{ are well-defined and } \\
&\T_{\R}(f(u), g(v)) \neq \varnothing \bigr\} \,.
\end{align*}
\end{enumerate}
\end{defn}

\begin{remark}
\leavevmode
\begin{enumerate}
\item
Note that we have $P_3(f, g) \subseteq P_2(f, g) \subseteq P_1(f, g)$ for any
two power series $f,g \in \R(a)[[x_1,\ldots,x_n]]$.

\item
For any two unimodal singularity (sub-)types $T_1, T_2$ and $i \in \{1,2,3\}$,
we simply write $P_i(T_1,T_2)$ instead of $P_i(\NF(T_1(a)), \NF(T_2(a)))$,
e.g.\@ we write $P_1\left(E_{14}^+, E_{14}^+\right)$ for
$P_1\left(\NF\left(E_{14}^+(a)\right), \NF\left(E_{14}^+(a)\right)\right)$.
\end{enumerate}
\end{remark}

For the parabolic singularity types $X_9$ and $J_{10}$, the sets $P_1$, $P_2$,
and $P_3$ can be described in terms of the following definition.

\begin{defn}
For $\Omega \subset \C$, let $(f_i: \Omega \rightarrow \C)_{i \in I}$ be a
family of complex-valued functions on $\Omega$. We define the joint graph of
$(f_i)_{i \in I}$ over $\Omega$ as
\[
\Gamma_\Omega((f_i)_{i \in I})
:= \{ (a, f_i(a)) \in \Omega \times \C \mid a\in \Omega,\; i \in I \}\,.
\]
\end{defn}

It turns out that for the hyperbolic and exceptional unimodal singularities,
$P_1$, $P_2$ and $P_3$ are just unions of sets of the form $(a, ra)_{a \in \K}$
for some $r \in \K$. For those cases we use the following notations.

\begin{defn}\label{def:C_R}
For any polynomial $p(X) \in \C[X]$, we define the sets $C_0(p(X))$ and
$R_0(p(X))$ as
\begin{align*}
C_0(p(X)) &:= \{ (a, ra) \in \C^2 \mid a, r \in \C, \; p(r) = 0 \} \,, \\
R_0(p(X)) &:= \{ (a, ra) \in \R^2 \mid a, r \in \R, \; p(r) = 0 \} \,.
\end{align*}
Additionally, we define $C(p(X))$ and $R(p(X))$ as
\begin{align*}
C(p(X)) &:= C_0(p(X)) \setminus \{(0, 0)\} \,, \\
R(p(X)) &:= R_0(p(X)) \setminus \{(0, 0)\} \,.
\end{align*}
\end{defn}

\begin{remark}
We occasionally use the notation $R\left(X^l-s\right)$ with
$l \in \N \setminus \{0\}$ and $s \in \{-1, +1\}$, e.g.\@ in
Tables~\ref{tab:J10+k_equivalences} and \ref{tab:X9+k_equivalences}. Of course,
this could be written in a more explicit way for many values of $l$ and $s$;
for instance, we could write $\varnothing$ instead of $R\left(X^4+1\right)$.
But distinguishing between different cases would spoil the symmetries of those
tables and we therefore stick to the shorthand notation.
\end{remark}

\section{Weighted Jets and Filtrations of Power Series and Transformations}%
\label{sec:weights}

We briefly introduce the concepts of (piecewise) weighted jets and filtrations.
For background regarding the definitions in this section, we refer to
\citet{A1974}. We assume that the reader is familiar with the notions of
weighted degrees, quasihomogeneous polynomials, and Newton polygons.

\begin{remark}
Let $w$ be a weight on the variables $(x_1, \ldots, x_n)$. Throughout this
paper we always assume that the weighted degree of $x_i$, denoted by
$w\dash\deg(x_i)$, is a natural number for each $i = 1, \ldots, n$.
\end{remark}

\begin{defn}
Let $w_0 := (w_1, \ldots, w_s) \in \left(\N^n\right)^s$ be a finite family of
weights on the variables $(x_1, \ldots, x_n)$. For any term
$t \in \K[x_1,\ldots,x_n]$, we define the piecewise weight of $t$ w.r.t.\@
$w_0$ as
\[
w_0\dash\deg(t) := \min_{i = 1, \ldots, s} w_i\dash\deg(t) \,.
\]
A polynomial $f \in \K[x_1,\ldots,x_n]$ is called piecewise quasihomogeneous of
degree $d$ w.r.t.\@ $w_0$ if $w_0\dash\deg(t) = d$ for any term $t$ of $f$.
\end{defn}

\begin{defn}
Let $w$ be a (piecewise) weight on the variables $(x_1,\ldots,x_n)$.

\begin{enumerate}
\item
Let $f = \sum_{i = 0}^{\infty} f_i$ be the decomposition of
$f \in \K[[x_1,\ldots,x_n]]$ into weighted homogeneous parts $f_i$ of
$w$-degree $i$. We denote the weighted $j$-jet of $f$ w.r.t.\@ $w$ by
\[
w \dash \jet(f, j) := \sum_{i = 0}^j f_i \,.
\]

\item
A power series in $\K[[x_1,\ldots,x_n]]$ has filtration $d \in \N$ if all its
terms are of weighted degree $d$ or higher. The power series of filtration $d$
form a vector space $E_d^w \subset \K[[x_1,\ldots,x_n]]$.
\end{enumerate}
\end{defn}

\begin{remark}
Note that $d < d'$ implies $E_{d'}^w \subseteq E_d^w$. Since the filtration of
the product $E_{d'}^w \cdot E_d^w$ is $d'+d$, it follows that $E_d^w$ is an
ideal in the ring of power series. We denote the ideal consisting of power
series of filtration strictly greater than $d$ by $E_{>d}^w$. If the weight
of each variable is $1$, we simply write $E_d$ and $E_{>d}$, respectively.
\end{remark}

There are also similar concepts for coordinate transformations:

\begin{defn}\label{def:jet_phi}
Let $\phi$ be a $\K$-algebra automorphism of $\K[[x_1,\ldots,x_n]]$ and let
$w$ be a (piecewise) weight on the variables.

\begin{enumerate}
\item
For $j > 0$ we define the weighted $j$-jet of $\phi$ w.r.t.\@ $w$, denoted by
$\phi_j^w$, to be the map given by
\[
\phi_j^w(x_i) := w\dash\jet(\phi(x_i),w\dash\deg(x_i)+j) \quad
\forall i = 1,\ldots,n \,.
\]
If the weight of each variable is $1$, i.e.\@ $w = (1, \ldots, 1)$, we simply
write $\phi_j$ for $\phi_j^w$.

\item\label{enum:filtration}
$\phi$ has filtration $d$ if, for all $\lambda \in \N$,
\[
(\phi-\id)E_\lambda^w \subset E_{\lambda+d}^w \,.
\]
\end{enumerate}
\end{defn}

\begin{remark}\label{rem:weighted_jet}
Let $\phi$ be a $\K$-algebra automorphism of $\K[[x_1,\ldots,x_n]]$.

\begin{enumerate}
\item
Note that $\phi_0(x_i) = \jet(\phi(x_i), 1)$ for all $i = 1, \ldots, n$.
Furthermore note that $\phi_0^w$ may have filtration less than or equal to $0$
for any weight $w$.

\item\label{enum:weighted_jet}
Let $w_0 = (w_1, \ldots, w_s) \in \left(\N^n\right)^s$ be a piecewise weight on
$(x_1, \ldots, x_n)$, let $f_0 \in \K[x_1,\ldots,x_n]$ be piecewise
quasihomogeneous of degree $d_0$ w.r.t.\@ $w_0$ and
$f_1 \in \K[x_1,\ldots,x_n]$ quasihomogeneous of degree $d_1$ w.r.t.\@
$w_1$. For any $\delta \geq 0$, we always have
$(\phi-\phi_\delta^{w_1})(f_1) \in E_{>d_1+\delta}^{w_1}$, but the analogon for
$w_0$ does not hold in general: To give a counterexample, let us consider the
case $n = s = 2$, $w_0 = ((1,4), (4,1))$, $f_0 = x_1 x_2$, and let $\phi$ be
given by $\phi(x_1) := x_1+x_2^2$, $\phi(x_2) := x_2+x_2^2$. Then $f_0$ is of
degree $d_0 = 5$, but $(\phi-\phi_0^{w_0})(f_0) = x_2^4$ is of degree $4$
w.r.t.\@ $w_0$ and thus not an element of $E_5^{w_0} = E_{d_0+0}^{w_0}$.
\end{enumerate}
\end{remark}

\section{Sufficient Sets of Transformations}\label{sec:sufficient_sets}

The results in this section considerably narrow down the transformations we
need to consider between specific unimodal normal forms of the same main type
in order to check if they are equivalent or not. In fact these results are in
many cases the main step for determining the structure of the equivalence
classes of the unimodal singularities up to corank $2$.

\begin{defn}
Let $f$ and $g$ be elements in $\C(a)[[x_1,\ldots,x_n]]$ and let $S$ be a
subset of $\Aut_{\C}(\C[[x_1,\ldots,x_n]])$. We call $S$ a sufficient set of
coordinate transformations for the pair $(f, g)$ if
\[
\forall u,v \in \C: \quad
\left(\T_{\C}(f(u),g(v)) \neq \varnothing
\;\Leftrightarrow\; \T_{\C}^S(f(u),g(v)) \neq \varnothing\right) \,.
\]
\end{defn}

The sufficient sets which we consider here can be described using the
following notation.

\begin{defn}
Let $M_x$ and $M_y$ be sets of monomials in $\C[[x,y]]$ and let $\C\! M_x$ and
$\C\! M_y$ be the $\C$-vector spaces spanned by these sets, i.e.\@
$\C\! M_x := \bigoplus_{m \in M_x} \C m$ and analogously for $\C\! M_y$. We
define the set of coordinate transformations spanned by $M_x$ and $M_y$ as
\[
\s(M_x, M_y) := \{ \phi \in \Aut_{\C}(\C[[x,y]]) \mid
\phi(x) \in \C\! M_x,\; \phi(y) \in \C\! M_y \} \,.
\]
\end{defn}

\begin{theorem}\label{thm:sufficient_sets}
Let $T$ be one of the main singularity types listed in
Table~\ref{tab:sufficient_sets}, let $S$ be the corresponding set of
automorphisms, and let $T_1$ and $T_2$ be subtypes of $T$. Then $S$ is a
sufficient set of coordinate transformations for
$\bigl(\NF(T_1(a)), \NF(T_2(a))\bigr)$.

\begin{table}[htb]
\centering
\caption{Sufficient sets for unimodal singularities of corank~2}
\label{tab:sufficient_sets}
\begin{tabular}{|c|c|c|c|}
\hline

& \multicolumn{2}{|c|}{$T$} & $S$ \\
\hline\hline

\multirow{2}{*}{\begin{sideways}P.\end{sideways}}

& \multicolumn{2}{|c|}{$X_9$}    & $\s(\{x, y\}, \{x, y\})$ \\ \cline{2-4}
& \multicolumn{2}{|c|}{$J_{10}$} & $\s(\{x, y^2\}, \{y\})$  \\ \hline

\multirow{4}{*}{\begin{sideways}Hyperbolic\end{sideways}}

& \multicolumn{2}{|c|}{$J_{10+k}$} & $\s(\{x\}, \{y\})$ \\ \cline{2-4}
& \multicolumn{2}{|c|}{$X_{9+k}$}  & $\s(\{x\}, \{y\})$ \\ \cline{2-4}

& \multirow{2}{*}{$Y_{r,s}$}
    & $r \neq s$ & $\s(\{x\}, \{y\})$                       \\ \cline{3-4}
&   & $r = s$    & $\s(\{x\}, \{y\}) \cup \s(\{y\}, \{x\})$ \\ \hline

\multirow{3}{*}{\begin{sideways}Except.\end{sideways}}

& \multicolumn{2}{|c|}{$E_{12}, E_{13}, E_{14}$} & $\s(\{x\}, \{y\})$
\\ \cline{2-4}
& \multicolumn{2}{|c|}{$Z_{11}, Z_{12}, Z_{13}$} & $\s(\{x\}, \{y\})$
\\ \cline{2-4}
& \multicolumn{2}{|c|}{$W_{12}, W_{13}$}         & $\s(\{x\}, \{y\})$
\\ \hline

\end{tabular}
\end{table}

\end{theorem}

\begin{proof}[Proof of Theorem~\ref{thm:sufficient_sets}]
We give different proofs for the parabolic, the hyperbolic, and the exceptional
cases as indicated in Table~\ref{tab:sufficient_sets}.

In each case, let $T_1$ and $T_2$ be subtypes of the same main type $T$, and
for $u \in \C$ let $\phi \in \Aut_{\C}(\C[[x,y]])$ be a coordinate
transformation which takes $f := \NF(T_1(u))$ to $\NF(T_2(v))$ for some
$v \in \C$.

\begin{description}[leftmargin=0cm,labelindent=\parindent,
before={\renewcommand\makelabel[1]{\normalfont\itshape ##1.}}]

\item[Parabolic cases]
The normal forms of both $X_9$ and $J_{10}$ are
quasihomogeneous with weights $w := (1,1)$ and $w := (2,1)$, respectively. Let
us first consider the case $T = X_9$. We have
\[
\phi(f) = \phi_0^w(f)+(\phi-\phi_0^w)(f) = \phi_0^w(f)+R
\]
with $R \in E_{>4}^w$. This implies $\phi(f) = \phi_0^w(f)$ because
$\phi(f) = \NF(T_2(v))$ is homogeneous of degree $4$ w.r.t.\@ the weight $w$.
So any possible value of $v$ which can be reached via some
$\phi \in \Aut_{\C}(\C[[x,y]])$ can also be obtained by
$\phi_0^w \in \s\left(\{x,y\}, \{x,y\}\right)$, i.e.\@,
$\s\left(\{x,y\}, \{x,y\}\right)$ is a sufficient set of coordinate
transformations for the pair $\left(\NF(T_1(a)), \NF(T_2(a))\right)$.

Let us now consider the case $T = J_{10}$.
Again we have $\phi(f) = \phi_0^w(f)$, but
in this case
$\phi_0^w$ is of the form
$\phi_0^w(x) = \alpha x + \beta y + \gamma y^2$,
$\phi_0^w(y) = \delta y$ with
$\alpha, \beta, \gamma, \delta \in \C$. Comparing the coefficients of
$\phi(f) = \NF(T_2(v))$ and
$\phi_0^w(f) = \beta^3 y^3 + (\text{other terms})$ yields
$\beta = 0$ and therefore $\phi_0^w \in \s(\{x, y^2\}, \{y\})$
as expected.

\item[Hyperbolic cases]
We present a proof for the main type $T = J_{10+k}$, the proofs for $X_{9+k}$
and $Y_{r,s}$ are similar. For $Y_{r,s}$ with $r = s$, we have to take the
special shape of $S$ into account, cf.\@ Table~\ref{tab:sufficient_sets}.

It does not matter for the arguments below whether we assume $T_1 = J_{10+k}^+$
or $T_1 = J_{10+k}^-$, the same holds for $T_2$. We write $\pm$ whenever the
sign can be either plus or minus in order to prove all cases at once.

The Newton polygon of $f = \NF(T_1(u)) = x^3 \pm x^2 y^2 + uy^{6+k}$ has two
faces defined by $f_1 := x^3 \pm x^2 y^2$ and $f_2 := \pm x^2 y^2 + uy^{6+k}$.
Let $w_0$ be the piecewise weight given by the two weights
$w_1 := (12+2k, 6+k)$ and $w_2 := (12+3k, 6)$. Then $f$ is piecewise
quasihomogeneous of degree $d := 36+6k$ w.r.t.\@ $w_0$.

We now proceed in three steps: In the first two, we show
$\phi_0^{w_1} \in \s(\{x\}, \{y\})$ and $\phi_0^{w_2} \in \s(\{x\}, \{y\})$.
Finally we conclude that $\phi(f)$ is equal to $\phi_0^{w_0}(f)$ and that
$\phi_0^{w_0}$ is an element of $\s(\{x\}, \{y\})$ which proves the claim.

The transformation $\phi_0^{w_1}$ is generically of the form
\begin{align*}
\phi_0^{w_1}(x) &= \alpha x + \beta_1 y + \beta_2 y^2 \,, \\
\phi_0^{w_1}(y) &= \gamma y
\end{align*}
with coefficients $\alpha, \beta_1, \beta_2, \gamma \in \C$. With these
notations we have
\begin{align*}
\phi(f)
&= \phi_0^{w_1}(f) + (\phi-\phi_0^{w_1})(f) \\
&= \beta_1^3 y^3 + (3\alpha\beta_2^2 \pm 2\alpha\beta_2\gamma^2) xy^4
+ (\beta_2^3 \pm \beta_2^2\gamma^2) y^6 + (\text{other terms})
\end{align*}
on the one hand and
\begin{align*}
\phi(f)
= \NF(T_1(v))
= x^3 \pm x^2 y^2 + vy^{6+k}
\end{align*}
on the other hand. This implies (in this order) $\beta_1 = 0$, $\alpha \neq 0$,
$\beta_2 = 0$ and hence $\phi_0^{w_1} \in \s(\{x\}, \{y\})$.

The second step is a proof by contradiction. Let $m$ be the largest integer
which is not greater than $\frac{k}{2}+2$. Then similar as above, the
automorphism $\phi_0^{w_2}$ is of the form
\begin{align*}
\phi_0^{w_2}(x)
&= \alpha x + \beta_1 y + \beta_2 y^2 + \ldots + \beta_m y^m \,, \\
\phi_0^{w_2}(y) &= \gamma y
\end{align*}
with $\alpha, \beta_1, \ldots, \beta_m, \gamma \in \C$ and
$\alpha, \gamma \neq 0$. We have already shown $\beta_1 = \beta_2 = 0$. Assume
$\beta_s \neq 0$ for some $s \in \{3, \ldots, m\}$ and let $s$ be minimal with
this property. Then the coefficient of $xy^{s+2}$ in
$\phi(f) = \phi_0^{w_2}(f) + (\phi-\phi_0^{w_2})(f)$ is
$\pm 2\alpha\beta_s\gamma^2$ which implies $\beta_s = 0$ in contradiction to
the assumption. Hence $\beta_3 = \ldots = \beta_m = 0$ and
$\phi_0^{w_2} \in \s(\{x\}, \{y\})$.

For the last step, we consider the following equations:
\begin{align*}
\phi(f) &= \phi_0^{w_1}(f)
+\myunderbrace{(\phi-\phi_0^{w_1})(f)}{=}{: R_1 \in E_{>d}^{w_1}} \\[.5ex]
\phi(f) &= \phi_0^{w_2}(f)
+\myunderbrace{(\phi-\phi_0^{w_2})(f)}{=}{: R_2 \in E_{>d}^{w_2}}
\displaybreak[0] \\[.5ex]
\phi(f) &= \phi_0^{w_0}(f)
+\myunderbrace{(\phi-\phi_0^{w_0})(f)}{=}{: R_0}
\end{align*}
Note that it is not a priori clear that $R_0$ lies in $E_{>d}^{w_0}$ if we only
consider these equations, cf.\@
Remark~\ref{rem:weighted_jet}(\ref{enum:weighted_jet}). Nevertheless, this can
be shown if we take into account the results of the two previous steps: By
definition of the piecewise weight $w_0$, any term in $\phi_0^{w_0}(x)$ also
appears in $\phi_0^{w_1}(x)$ or $\phi_0^{w_2}(x)$ (or both), analogously for
$\phi_0^{w_0}(y)$. Therefore we have $\phi_0^{w_0}(x) = \alpha x$ and
$\phi_0^{w_0}(y) = \gamma y$, hence
$\phi_0^{w_0} = \phi_0^{w_1} = \phi_0^{w_2}$ and
$\phi_0^{w_0} \in \s(\{x\}, \{y\})$. This implies
\[
R_0 = R_1 = R_2 \in E_{>d}^{w_1} \cap E_{>d}^{w_2} = E_{>d}^{w_0} \,.
\]
Since $\phi(f) = \NF(T_2(v))$ is piecewise quasihomogeneous of degree $d$
w.r.t.\@ $w_0$, we finally get $R_0 = 0$ and $\phi(f) = \phi_0^{w_0}(f)$. This
proves the claim.

\item[Exceptional cases]
The normal forms of all the exceptional cases in
Table~\ref{tab:sufficient_sets} are semi-quasihomogeneous polynomials, i.e., in
these cases $f = \NF(T_1(u))$ is of the form $f = f_0+f_1$ where $f_0$ is
quasihomogeneous of degree $d \in \N$ w.r.t.\@ some weight $w = (w_x, w_y)$,
$f_1$ has weighted degree $d+\delta > d$, and the Milnor number $\mu(f_0)$ of
$f_0$ is finite \citep[for the definition of the Milnor number, see][]{MS2015}.
In all the cases, $f_1$ consists of the term which contains the parameter and
we have
\begin{align*}
\phi(f) &= \phi_\delta^w(f_0) + \phi_0^w(f_1)
+ \left(\phi-\phi_\delta^w\right)(f_0) + \left(\phi-\phi_0^w\right)(f_1) \\
&= w\dash\jet(\phi_\delta^w(f_0), d+\delta) + \phi_0^w(f_1) + R
\end{align*}
with $R \in E_{>d+\delta}^w$. As above, $\phi(f) = \NF(T_2(v))$ implies
$R = 0$. If we show
\begin{align}
\phi_\delta^w \in \s\left(\{x\}, \{y\}\right) \,, \tag{$\ast$}
\end{align}
then it follows that $\phi_\delta^w$ is equal to $\phi_0^w$ and therefore
\begin{align*}
\phi(f)
&= w\dash\jet(\phi_0^w(f_0), d+\delta) + \phi_0^w(f_1) \\
&= \phi_0^w(f_0)+\phi_0^w(f_1) \\
&= \phi_\delta^w(f) \,.
\end{align*}
This, together with $(\ast)$, proves the claim.

The statement $(\ast)$ can be shown separately for each of the eight cases by
some easy computations. We carry out the proof for $W_{13}$, the other cases
follow similarly. The normal forms of the subtypes of $W_{13}$ are
$\pm x^4+xy^4+ay^6$, so in this case we have $w = (4,3)$, $d = 16$, and
$\delta = 2$. The $\pm$-sign does not matter for the computations which follow,
but we carry it along in order to prove all subcases at once. The
transformation $\phi_\delta^w$ is generically of the form
\begin{align*}
\phi_\delta^w(x) &= \alpha x + \beta y + \gamma y^2 \,, \\
\phi_\delta^w(y) &= \varepsilon x + \zeta y
\end{align*}
with $\alpha, \beta, \gamma, \varepsilon, \zeta \in \C$ because any other term
would raise the weighted degree by more than $\delta$. With these notations, we
now successively compare the coefficients of $\phi_\delta^w(f)$ and
$\phi(f) = \NF(T_2(v)) = \pm x^4+xy^4+vy^6$. The coefficient of $y^4$ in
$\phi_\delta^w(f)$ is $\pm \beta^4$, therefore we have $\beta = 0$. The
remaining coefficients of $xy^4$, $x^2y^3$, and $x^3y^2$ are now
$\alpha\zeta^4$, $4\alpha\varepsilon\zeta^3$, and
$\pm 4\alpha^3\gamma+6\alpha\varepsilon^2\zeta^2$, respectively, which shows
that (in this order) $\alpha\zeta \neq 0$, $\varepsilon = 0$, and $\gamma = 0$.
Hence $\phi_\delta^w$ is in fact of the form $\phi_\delta^w(x) = \alpha x$,
$\phi_\delta^w(y) = \zeta y$ which proves $(\ast)$ for
$T_1, T_2 \in \left\{W_{13}^+, W_{13}^-\right\}$.
\end{description}
\end{proof}

\section{On the Computation of the Results}\label{sec:computations}

Based on the previous section, the results presented in
Section~\ref{sec:results} can be computed using \Singular{} for all those
singularity types which are covered by Theorem~\ref{thm:sufficient_sets}. The
main tools for these computations are elimination, Gr\"obner covers, and
primary decomposition. For details on these topics, we refer to \citet{GP2008}.
For each pair of singularity subtypes~$T_1, T_2$, the computation follows the
same structure: One can first compute the set $P_1(T_1, T_2)$ using elimination
and factorization. The set $P_2(T_1, T_2)$ can then be derived from this as the
intersection of $P_1(T_1, T_2)$ with $\R \times \R$. In order to determine
$P_3(T_1, T_2)$, one finally has to check for each point or branch in
$P_2(T_1, T_2)$ whether or not there is a real transformation which changes the
parameter in such a way. Gr\"obner covers and primary decomposition are
convenient tools to simplify the often complicated ideals which occur in this
last step.

Although our approach is almost algorithmic, we do not present it as an
algorithm here because each case requires slightly different means depending on
the intermediate results. Especially the computation of $P_3(T_1, T_2)$ is
rather straightforward in some cases whereas it requires careful considerations
in other cases.

However, writing down every detail of the computations for each case is beyond
the scope of this section. Instead, we present the general framework and give
explicit \Singular{} commands for $T_1 = T_2 = X_9^{++}$ which is one of the
more complicated cases (cf. Theorem~\ref{thm:X9}).

The singularity type $\tY_r$ does not appear in Table~\ref{tab:sufficient_sets}
and thus needs special care. The structure of the equivalence classes of this
type can be computed on the basis of the data for the type $Y_{r,s}$, cf.\@
Section~\ref{ssec:Yr}.

\subsection{How to Compute
\texorpdfstring{$\boldsymbol{P_1(T_1, T_2)}$}{P1(T1, T2)}}%
\label{ssec:computing_P1}

We denote the parameter occurring in $\NF(T_1)$ by $a$ and the one in
$\NF(T_2)$ by $b$. The computation is done in four steps:

\begin{description}[font=\normalfont\itshape,leftmargin=0cm,
labelindent=\parindent,style=nextline]

\item[Step 1. Set up a generic transformation using
Theorem~\ref{thm:sufficient_sets}]
Let $S = \s(M_x, M_y) \subset \Aut_{\C}(\C[[x,y]])$ be the sufficient set of
$\C[[x,y]]$-automor\-phisms for $(\NF(T_1(a)), \NF(T_2(a)))$ given in
Theorem~\ref{thm:sufficient_sets}. Let $t_1, \ldots, t_r$ be coefficients for
the monomials in $M_x$ and $M_y$ and let $\phi$ be a generic element of $S$
with these coefficients, i.e.\@ let $\phi$ be of the form
$\phi(x) = t_1 \cdot x + (\text{other terms})$ (or of the form
$\phi(x) = t_1 \cdot y + (\text{other terms})$ in case $T_1$ and $T_2$ are
subtypes of $Y_{r,s}$ with $r = s$).

\item[Step 2. Set up a system of equations for the parameters]
By comparing the coefficients in $\phi(\NF(T_1(a)))$ and $\NF(T_2(b))$, we get
a set of equations in $a, b, t_1, \ldots, t_r$ which is equivalent to
$\phi(\NF(T_1(a))) = \NF(T_2(b))$. Let $I \subset \C[a,b,t_1,\ldots,t_r]$ be
the ideal generated by these equations. Then the vanishing set $V(I)$ describes
completely which transformations take $\NF(T_1(a))$ to $\NF(T_2(b))$ for which
values of $a$ and $b$.

\item[Step 3. Use elimination to obtain an equation in $a$ and $b$ only]
We can now eliminate the variables $t_1, \ldots, t_r$ from $I$ and thus obtain
an ideal $I' \subset \C[a,b]$ which is in all cases generated by one polynomial
$g$. This elimination geometrically corresponds to the projection
$\A_{\C}^{2+r} \supset V(I) \mapsto V(I') \subset \A_{\C}^2$.

\item[Step 4. Compute the zeros of this equation]
After factorizing $g \in \C[a,b]$ into irreducible factors $g_1, \ldots, g_s$,
we compute the roots in $b$ of each factor (over $\C(a)$ or suitable extensions
thereof if necessary). We thus get roots of the form $b-f(a)$ where $f(a)$ can
be considered as a function in $a$. These functions explicitly determine the
possible values of $b$ for each given $a$ and their joint graph is exactly
$P_1(T_1, T_2)$.
\end{description}

\begin{example}\label{ex:P1}
We compute $P_1\left(X_9^{++}, X_9^{++}\right)$ with Singular, following the
steps above.

\begin{description}[font=\normalfont\itshape,leftmargin=0cm,
labelindent=\parindent,style=nextline]

\item[Step 1. Set up a generic transformation using
Theorem~\ref{thm:sufficient_sets}]
For convenience we work over $\Q(a,b,t_1,t_2,t_3,t_4)[x,y]$:
\begin{verbatim}
> ring R = (0,a,b,t1,t2,t3,t4), (x,y), dp;
> poly f = x^4+a*x^2*y^2+y^4;
\end{verbatim}
According to Theorem~\ref{thm:sufficient_sets},
\[
S = \bigl\{ \phi \in \Aut_{\C}(\C[[x,y]])
\mid \phi(x) = t_1 x + t_2 y,\; \phi(y) = t_3 x + t_4 y,\;
t_1, \ldots, t_4 \in \C \bigr\}
\]
is a sufficient set of automorphisms for $X_9$:
\begin{verbatim}
> map phi = R, t1*x+t2*y, t3*x+t4*y;
\end{verbatim}

\item[Step 2. Set up a system of equations for the parameters]
\mbox{}\vspace{-\baselineskip}\vspace{-\topsep}
\begin{verbatim}
> matrix C = coef(phi(f), xy);
> print(C);
x^4,   x^3*y, x^2*y^2,x*y^3, y^4,  
C[2,1],C[2,2],C[2,3], C[2,4],C[2,5]
> C[2,1];
(a*t1^2*t3^2+t1^4+t3^4)
\end{verbatim}
Now the second row of the matrix \verb+C+ contains the coefficients of
$\verb+phi+\left(X_9^{++}(a)\right)$, \verb+C[2, 1]+ for instance is the one
belonging to $x^4$. Using the corresponding coefficients of
$X_9^{++}(b) = x^4 + b \cdot x^2 y^2 + y^4$, we can define the ideal $I$ as
above:
\begin{verbatim}
> matrix D[1][5] = 1, 0, b, 0, 1;
> ideal I = C[2,1..5]-D[1,1..5];
\end{verbatim}

\item[Step 3. Use elimination to obtain an equation in $a$ and $b$ only]
As the next step, we map this ideal to $\Q(a)[b,t_1,t_2,t_3,t_4]$ and eliminate
the variables~$t_i$:
\begin{verbatim}
> ring S = (0,a), (b,t1,t2,t3,t4), dp;
> ideal I = imap(R, I);
> ideal g = eliminate(I, t1*t2*t3*t4);
> g;
g[1]=(a^4-8*a^2+16)*b^6+(-a^6-720*a^2-1152)*b^4
+(8*a^6+720*a^4+20736)*b^2+(-16*a^6+1152*a^4-20736*a^2)
\end{verbatim}

\item[Step 4. Compute the zeros of this equation]
Factorizing the single generator of this ideal finally yields the functions
$f_1^{1,1}, \ldots, f_6^{1,1}$ defined in Theorem~\ref{thm:X9}. Note that
$a^2 \neq 4$.
\begin{verbatim}
> factorize(g[1]);
[1]:
   _[1]=1
   _[2]=b+(-a)
   _[3]=b+(a)
   _[4]=(a-2)*b+(-2*a-12)
   _[5]=(a+2)*b+(-2*a+12)
   _[6]=(a+2)*b+(2*a-12)
   _[7]=(a-2)*b+(2*a+12)
[2]:
   1,1,1,1,1,1,1
\end{verbatim}
\end{description}
\end{example}

\subsection{How to Compute
\texorpdfstring{$\boldsymbol{P_2(T_1, T_2)}$}{P2(T1, T2)}}%
\label{ssec:computing_P2}

Given $P_1(T_1, T_2)$, it is easy to compute $P_2(T_1, T_2)$ even ``by hand''
because we have
\[
P_2(T_1, T_2) = P_1(T_1, T_2) \cap (\R \times \R) \,.
\]

\begin{example}\label{ex:P2}
Continuing the example above, the values of
$f_1^{1,1}(a), \ldots, f_6^{1,1}(a)$ are clearly real for $a \in \R$, cf.\@
Theorem~\ref{thm:X9}. The set $P_2\left(X_9^{++}, X_9^{++}\right)$ is thus the
joint graph of these functions over $\R \setminus \{-2, 2\}$.

To give another example, for $T_1 = T_2 = X_9^{+-}$ the set $P_1(T_1, T_2)$ is
the joint graph of $f_1^{i,i}, \ldots, f_6^{i,i}$ over
$\C \setminus \{-2i, 2i\}$. The values of $f_1^{i,i}(a)$ and $f_2^{i,i}(a)$ are
clearly real for $a \in \R$, but those of $f_3^{i,i}(a), \ldots, f_6^{i,i}(a)$
are not except at some exceptional points which are already covered by
$f_1^{i,i}$ and $f_2^{i,i}$. So in this case we have
\[
P_2\left(X_9^{+-}, X_9^{+-}\right)
= \Gamma_{\R} \left(f_1^{i,i}, f_2^{i,i}\right)
= \Gamma_{\R} \left(f_1^{1,1}, f_2^{1,1}\right) \,.
\]
\end{example}

\subsection{How to Compute
\texorpdfstring{$\boldsymbol{P_3(T_1, T_2)}$}{P3(T1, T2)}}%
\label{ssec:computing_P3}

We do the computation in two steps:

\begin{description}[font=\normalfont\itshape,leftmargin=0cm,
labelindent=\parindent,style=nextline]

\item[Step 1. Reduce the problem to a finite number of branches and exceptional
points]
Since $P_3(T_1, T_2) \subset P_2(T_1, T_2)$ by definition, we can determine
$P_3(T_1, T_2)$ by checking for each pair $(a,b) \in P_2(T_1, T_2)$ whether or
not there is a \emph{real} coordinate transformation
$\phi \in \Aut_{\R}(\R[[x,y]])$ which takes $\NF(T_1(a))$ to $\NF(T_2(b))$.
This can be reduced to a finite problem as follows: Let $g_j$, $j \in
\{1,\ldots,s\}$, be the irreducible factors of the polynomial $g$ as in
Section~\ref{ssec:computing_P1}. Then in all the cases, $P_2(T_1, T_2)$ is a
finite union of ``branches'' of the form $V(g_j)$ and some exceptional points.
We can check whether a branch $V(g_j)$ or an exceptional point $(q_a, q_b)$ in
$P_2(T_1, T_2)$ belongs $P_3(T_1, T_2)$ by simply adding appropriate relations
to the ideal~$I$ and looking at the real solutions of the resulting ideal. In
other words, we define $J := I+\langle g_j \rangle$ or
$J := I+\langle a-q_a, b-q_b \rangle$, respectively, and investigate
$V_{\R}(J)$. Note that we have $I \subset \R[a,b,t_1,\ldots,t_r]$ and
$g_j \in \R[a,b]$ and thus $J \subset \R[a,b,t_1,\ldots,t_r]$ in all the cases.

\item[Step 2. For each branch and each exceptional point, check if a real
transformation exists]
$P_3(T_1, T_2)$ is the image of $V_{\R}(J) \subset \A_{\R}^{2+r}$ under the
projection $\A_{\R}^{2+r} \rightarrow \A_{\R}^2$, i.e.\@ we have
$(p_a, p_b) \in P_3(T_1, T_2)$ if and only if there is a coordinate
transformation with real coefficients $(p_{t_1}, \ldots, p_{t_r})$ such that
$(p_a, p_b, p_{t_1}, \ldots, p_{t_r})$ is an element of
$V_{\R}(J) \subset \A_{\R}^{2+r}$.

It turns out that the ideal $J$ is quite complicated in some cases and that it
can be difficult to determine $V_{\R}(J)$ by just computing a Gr\"obner basis
of $J$. One way out is then to consider $J$ as a parametric ideal
$J \subset \R(a)[b,t_1,\ldots,t_r]$ and to compute a Gr\"obner cover thereof
by using the \Singular{} library \verb+grobcov.lib+ \citep{grobcov}.
A~Gr\"obner cover completely describes the possible shapes of Gr\"obner bases
of $J$ for different values of $a$. It contains a generic Gr\"obner basis of
$J$, i.e.\@ one which is a Gr\"obner basis except for finitely many exceptional
values of $a$, and additionally Gr\"obner bases of $J$ for each of these
exceptional values. The ideals in a Gr\"obner cover of $J$ typically have a
much easier structure than $J$ itself. We can thus treat them one by one and
determine their real solutions. We will often find generators such as
$(t_j)^4+1$, indicating that the vanishing set over $\R$ of this ideal is
empty.

If any of the ideals in the Gr\"obner cover of $J$ are still to complicated and
if their vanishing set over $\R$ cannot be easily read off, another trick is to
compute a primary decomposition of these ideals with the \Singular{} library
\verb+primdec.lib+ \citep{primdec}. Typically, it is then easy to see that some
of the primary components have no solutions over $\R$ whereas the real
solutions of the remaining components can be easily determined.
\end{description}

\begin{example}
\leavevmode

\begin{description}[font=\normalfont\itshape,leftmargin=0cm,
labelindent=\parindent,style=nextline]

\item[Step 1. Reduce the problem to a finite number of branches and exceptional
points]
We have already seen in Example~\ref{ex:P2} that
$P_2\left(X_9^{++}, X_9^{++}\right)$ is the joint graph of
$f_1^{1,1}, \ldots,\allowbreak f_6^{1,1}$ over $\R \setminus \{-2, 2\}$. We now
have to check for each of these functions whether their graph is also contained
in $P_3\left(X_9^{++}, X_9^{++}\right)$.

\item[Step 2a. Check if a real transformation exists for $f_3^{1,1}$]
This is clearly the case for $f_1^{1,1} = \id$. To check this for $f_3^{1,1}$,
we continue the \Singular{} session from Example~\ref{ex:P1}, add the
corresponding relation to the ideal $I$ and compute a Gr\"obner cover of the
resulting ideal $J$:
\begin{verbatim}
> ideal J = I, (a-2)*b+(-2*a-12);
> LIB "grobcov.lib";
> grobcov(J);
\end{verbatim}
The output of the last command is too long to be printed here. We will find
that the Gr\"obner basis of $J$ for generic $a$ contains the generators
$(t_2)^2+(t_4)^2$ and $(t_3)^2+(t_4)^2$ which imply $t_2 = t_3 = t_4 = 0$ for
any real solution of this ideal. But this is a contradiction to
$\verb+phi+ \in \Aut_{\R}(\R[[x,y]])$. The exceptional cases for the
parameter~$a$ are $a+2=0$, $a-2=0$, $a^2+12=0$, $a+6=0$, $a-6=0$, and $a=0$.
The first two cases are excluded by the definition of the singularity type
$X_9^{++}$, $a^2+12=0$ would imply $a \not\in \R$, for $a+6=0$ and $a=0$ the
corresponding Gr\"obner bases of $J$ contain generators similar to those
mentioned above, and finally $a-6=0$ implies $b-6=0$ such that this case is
already covered by the graph of $f_1^{1,1}$.

\item[Step 2b. Check if a real transformation exists for $f_5^{1,1}$]
To give one more example, let us consider $f_5^{1,1}$:
\begin{verbatim}
> J = I, (a+2)*b+(2*a-12);
> grobcov(J);
\end{verbatim}
The crucial generator of the Gr\"obner basis of $J$ for generic $a$ is now the
polynomial $(a+2)(t_4)^4-1$ which has a real root if and only if $a > -2$.
Considering the other generators, it is easy to see that given $t_4 \in \R$,
$t_1 = t_2 = t_3 = -t_4$ is a real solution. The exceptional values
of $a$ in this case are the same as above and again, we do not have to consider
$a+2=0$, $a-2=0$, and $a^2+12=0$. The relation $a+6=0$ implies $b+6=0$ which is
already covered by $f_1^{1,1}$. Finally,
$t_1 = t_2 = t_3 = \frac{1}{\sqrt[4]{2}}$, $t_4 = -\frac{1}{\sqrt[4]{2}}$ and
$t_1 = t_2 = t_3 = \frac{1}{\sqrt[4]{8}}$, $t_4 = -\frac{1}{\sqrt[4]{8}}$ are
real solutions for the cases $a=0$ and $a-6=0$, respectively. To sum up,
the graph of $f_5^{1,1}$ over $\R^{>-2}$ belongs to
$P_3\left(X_9^{++}, X_9^{++}\right)$, but not the part over $\R^{<-2}$.

\item[Step 2c. Check if a real transformation exists for the other branches]
Continuing in this manner, one can show that $f_2^{1,1}$, $f_4^{1,1}$ and
$f_6^{1,1}$ do not contribute any additional points, so we get
\[
P_3\left(X_9^{++}, X_9^{++}\right) = \Gamma_{\R'} \left(f_1^{1,1}\right)
\cup \Gamma_{\R^{>-2}} \left(f_5^{1,1}\right)
\]
where $\R' := \R \setminus \{-2, 2\}$.
\end{description}
\end{example}

\begin{remark}
With the above notations, the irreducible factors $g_j$, $j = 1, \ldots, s$, of
the polynomial $g$ are luckily of degree $1$ in $b$ in almost all cases. If one
of those factors, say $g_1$, has degree in $b$ greater than $1$, and if
additionally the corresponding ideal $J = I + \langle g_1 \rangle$ has both
real and complex solutions, then an extra calculation is needed: Let
$f_1(a), \ldots, f_k(a)$ be the roots of $g_1$ in $b$ as above, i.e.\@
$g_1 = (b-f_1(a)) \ldots (b-f_k(a))$ (over $\C(a)$ or over a suitable extension
thereof if necessary). Then we have to check which of these roots
$f_1(a), \ldots, f_k(a)$ belong to the real solutions of $J$ and which of them
can only reached via complex transformations.

This is especially crucial for the singularities of type $J_{10}$ in order to
distinguish between $f_3^{\sigma,\rho}$, $f_4^{\sigma,\rho}$,
$f_5^{\sigma,\rho}$, and $f_6^{\sigma,\rho}$, cf.\@ Theorem~\ref{thm:J10}.
\end{remark}

\begin{remark}
The hyperbolic singularity types listed in Table~\ref{tab:normal_forms} are
actually infinite series of types. One might argue that the computations
described in Sections~\ref{ssec:computing_P1} to \ref{ssec:computing_P3} must
be carried out for each single $k > 0$ (for $J_{10+k}$ and $X_{9+k}$) and for
each pair $r,s > 4$ (for $Y_{r,s}$) in order to check the results presented in
Theorems~\ref{thm:J10+k} to \ref{thm:Yrs}. This is, of course, impossible in
practice. But it turns out that the results are periodic in $k$ and $r,s$,
respectively. Hence it suffices to carry these computations out for
sufficiently many values of $k$ and $r,s$. If we closely examine the
intermediate steps, then we can easily check that the results are indeed
periodic.
\end{remark}

\subsection{The Special Type \texorpdfstring{$\boldsymbol{\tY_r}$}{Yr}}%
\label{ssec:Yr}

Theorem~\ref{thm:sufficient_sets} does not give any sufficient set for subtypes
of $\tY_r$ and indeed it turns out that there is no degree-bounded sufficient
set for this case, cf.\@ Remark~\ref{rem:sufficient_sets_for_Yr}.

But since $\tY_r$ is $\C$-equivalent to $Y_{r,r}$, we can use the structure of
the equivalence classes of $Y_{r,r}$ (cf.\@ Theorem~\ref{thm:Yrs}) to determine
$P_1(T_1, T_2)$, $P_2(T_1, T_2)$, and $P_3(T_1, T_2)$ for
$T_1, T_2 \in \bigl\{\tY_r^+, \tY_r^-\bigr\}$. To do so, let us first define
the principal part of a power series.

\begin{defn}
Let $f \in \K[[x_1,\ldots,x_n]]$ be a power series, let $\Gamma_f$ be its
Newton polygon, and let $f_0$ be the sum of those terms of $f$ which lie on
$\Gamma_f$. Then we call $f_0$ the principal part of $f$.
\end{defn}

The following result is due to \citet[Corollary~9.9]{A1974}.

\begin{lemma}\label{lem:principalpart}
Let $f \in \C[[x,y]]$ be a power series whose principal part is of the form
$f_0 = x^a+\lambda x^2y^2+y^b$, where $0 \neq \lambda \in \C$, $a \geq 4$, and
$b \geq 5$. Then $f$ and its principal part $f_0$ are $\C$-equivalent, i.e.\@
$f \csim f_0$.
\end{lemma}

Based upon this lemma, we can now specify an explicit equivalence between the
normal forms of $\tY_r$ and $Y_{r,r}$.

\begin{lemma}\label{lem:Yr_equivalences}
For any $r > 4$ and any $a \in \C \setminus \{0\}$, we have
\[
\left(a, \left(\textstyle\frac{1}{4}\right)^r a^2\right)
\in P_1\left(\tY_r^+, Y_{r,r}^{++}\right)
\cap P_1\left(\tY_r^-, Y_{r,r}^{-+}\right) \,.
\]
\end{lemma}

\begin{proof}
Let $\phi \in \Aut_{\C}(\C[[x,y]])$ be the coordinate transformation defined by
$\phi(x) := \frac{1}{2}(x+y)$ and $\phi(y) := \frac{1}{2}i(x-y)$. Then the
principal parts of $\phi\bigl(\NF\bigl(\tY_r^+(a)\bigr)\bigr)$ and
$\phi\bigl(\NF\bigl(\tY_r^-(a)\bigr)\bigr)$ are of the form
$\left(\frac{1}{2}\right)^r a \cdot x^r + \lambda x^2 y^2
+ \left(\frac{1}{2}\right)^r a \cdot y^r$
with $\lambda = 1$ and $\lambda = -1$, respectively, so the result follows from
Lemma~\ref{lem:principalpart}.
\end{proof}

Section~\ref{ssec:computing_P1} tells us how to compute $P_1(T_1, T_2)$ for
$T_1, T_2 \in \{Y_{r,r}^{++}, Y_{r,r}^{-+}\}$, cf.\@ Theorem~\ref{thm:Yrs}. We
can use this data and the above lemma to compute $P_1(T_1, T_2)$ for
$T_1, T_2 \in \bigl\{\tY_r^+, \tY_r^-\bigr\}$. Let us consider the case
$P_1\bigl(\tY_r^+, \tY_r^+\bigr)$, the other cases follow similarly. According
to Lemma~\ref{lem:Yr_equivalences}, $\NF\bigl(\tY_r^+(a)\bigr)$ is
$\C$-equivalent to $\NF\bigl(Y_{r,r}^{++}\left(ca^2\right)\bigr)$ with
$c := \left(\frac{1}{4}\right)^r$ for any $r > 4$ and any
$a \in \C \setminus \{0\}$. This in turn is $\C$-equivalent to
$\NF\bigl(Y_{r,r}^{++}\left(\zeta ca^2\right)\bigr)$ for any $\zeta$ satisfying
$\zeta^l-1 = 0$ where $l = \gcd(2, r+1)$, cf.\@ Theorem~\ref{thm:Yrs}. Applying
Lemma~\ref{lem:Yr_equivalences} again leads to
$\NF\bigl(Y_{r,r}^{++}\left(\zeta ca^2\right)\bigr)
\csim \NF\bigl(\tY_r^+\left(\pm \sqrt{\zeta}a\right)\bigr)$,
and we thus get the diagram shown in Figure~\ref{fig:Yr}.

\begin{figure}
\caption{Equivalences between $\NF\bigl(\tY_r^+\bigr)$ and %
$\NF\left(Y_{r,r}^{++}\right)$\quad%
$\left(c := \left(\frac{1}{4}\right)^r\right)$}%
\label{fig:Yr}
\begin{center}
\begin{tikzpicture}
\matrix (m) [matrix of math nodes, row sep=1.5ex, column sep=2ex]
{
\NF\left(\tY_r^+\left(a\right)\right)                &&
\NF\left(\tY_r^+\left(\pm\sqrt{\zeta}a\right)\right) \\
& \text{\huge$\circlearrowleft$} &                   \\
\NF\left(Y_{r,r}^{++}\left(ca^2\right)\right)        &&
\NF\left(Y_{r,r}^{++}\left(\zeta ca^2\right)\right)  \\
};
\draw[<->] (m-1-1) -- (m-1-3);
\draw[<->] (m-1-3) -- (m-3-3);
\draw[<->] (m-3-3) -- (m-3-1);
\draw[<->] (m-3-1) -- (m-1-1);
\end{tikzpicture}
\end{center}
\end{figure}

This proves
$\NF\bigl(\tY_r^+(a)\bigr)
\csim \NF\bigl(\tY_r^+\left(\pm\sqrt{\zeta}a\right)\bigr)$
for $\zeta$ as above, and since the diagram is commutative, there are no
equivalences for other values of the parameters than these. Hence
\[
P_1\left(\tY_r^+, \tY_r^+\right) = C\left(X^{2l}-1\right)
\]
with $l$ as above. The set $P_2\bigl(\tY_r^+, \tY_r^+\bigr)$ can now be
determined as in Section~\ref{ssec:computing_P2}. In fact it is easy to see
that
\[
P_2\left(\tY_r^+, \tY_r^+\right) = R\left(X^2-1\right) \,.
\]
We clearly have $(a, a) \in P_3\bigl(\tY_r^+, \tY_r^+\bigr)$ for
$a \in \R \setminus \{0\}$, and also
$(a, -a) \in P_3\bigl(\tY_r^+, \tY_r^+\bigr)$ if $r$ is odd. For the case where
$r$ is even, let us consider $\NF\bigl(\tY_r^+(a)\bigr)$ as a function in $x$
and $y$ over $\R^2$ and let the parameter $a$ be positive. In this case the
function $\NF\bigl(\tY_r^+(a)\bigr) = \left(x^2+y^2\right)^2+ax^r$ takes only
non-negative values whereas
$\NF\bigl(\tY_r^+(-a)\bigr) = \left(x^2+y^2\right)^2-ax^r$ attains also
negative values. Hence there is no real coordinate transformation which takes
$\NF\bigl(\tY_r^+(a)\bigr)$ to $\NF\bigl(\tY_r^+(-a)\bigr)$. The argument is
similar for $a < 0$. To sum up, we have
\[
P_3\left(\tY_r^+, \tY_r^+\right) =
\begin{cases}
R(X^2-1), &\text{if } r \text{ is odd}, \\
R(X-1),   &\text{if } r \text{ is even}.
\end{cases}
\]

\begin{remark}\label{rem:sufficient_sets_for_Yr}
Let $r \geq 8$ be a multiple of $4$ and let $\phi_r \in \Aut_{\C}(\C[[x,y]])$
be a coordinate transformation which takes $f := \NF\bigl(\tY_r^+(a)\bigr)$ to
$\NF\bigl(\tY_r^+(-a)\bigr)$. Assume that the degree of both $\phi_r(x)$ and
$\phi_r(y)$ is less than $\frac{r}{4}$ and let $f = f_0 + f_1$ be decomposed
into its principal part $f_0 := \left(x^2+y^2\right)^2$ and $f_1 = ax^r$. Then
we have
\[
\phi(f) = \phi(f_0) + \phi(f_1) = \left(x^2+y^2\right)^2 - ax^r
\]
where the degree of $\phi(f_0)$ is less than $r$. Therefore
$\phi(f_0) = \phi_0(f_0) = \left(x^2+y^2\right)^2$ and
$\phi(f_1) = \phi_0(f_1) = -ax^r$. If $\phi_0$ is given by
$\phi_0(x) = \alpha x + \beta y$, $\phi_0(y) = \gamma x + \delta y$ with
$\alpha, \beta, \gamma, \delta \in \C$, the second of these two equations
implies $\beta = 0$ and $\alpha^r = -1$, but the first one in turn implies
$\gamma = 0$ and $\alpha^4 = 1$ which is a contradiction.

So the degree of either $\phi_r(x)$ or $\phi_r(y)$ must at least $\frac{r}{4}$.
This shows that a degree-bounded sufficient set of coordinate
transformations for
$\bigl(\NF\bigl(\tY_r^+(a)\bigr), \NF\bigl(\tY_r^+(a)\bigr)\bigr)$ and for
arbitrarily high $r$ does not exist.
\end{remark}

\section{Results}\label{sec:results}

In this section we present the sets $P_1,P_2,P_3$ in table form for every
unimodal real singularity type up to corank 2.

\begin{theorem}\label{thm:X9}
The structure of the equivalence classes of the $X_9$ singularities is as shown
in Table~\ref{tab:X9_equivalences} where for $j = 1, \ldots, 6$ and
$\rho, \sigma \in \{1, i\}$, the function $f_j^{\rho, \sigma}$ is defined as
follows:
\begin{align*}
f_1^{\rho, \sigma}(a) &:= +\rho \sigma \cdot a \,, &
f_3^{\rho, \sigma}(a) &:= \frac{+2\sigma a+12\rho\sigma}{a-2\rho} \,, &
f_5^{\rho, \sigma}(a) &:= \frac{-2\sigma a+12\rho\sigma}{a+2\rho} \,, \\
f_2^{\rho, \sigma}(a) &:= -\rho \sigma \cdot a \,, &
f_4^{\rho, \sigma}(a) &:= \frac{+2\sigma a-12\rho\sigma}{a+2\rho} \,, &
f_6^{\rho, \sigma}(a) &:= \frac{-2\sigma a-12\rho\sigma}{a-2\rho} \,.
\end{align*}

Furthermore, we use the following notations:
\begin{align*}
\C'  &:= \C \setminus \{ -2, 2\} \,, &
\R'  &:= \R \setminus \{ -2, 2\} \,, &
\C'' &:= \C \setminus \{ -2i, 2i\} \,.
\end{align*}

\begin{table}[!htbp]
\centering
\caption{$P_1$, $P_2$ and $P_3$ for the $X_9$ singularities}
\label{tab:X9_equivalences}
\begin{tabular}{|c|c||c|c|c|}
\hline

$T_1$ & $T_2$ & $P_1(T_1, T_2)$ & $P_2(T_1, T_2)$ & $P_3(T_1, T_2)$ \\
\hline\hline

$X_9^{++}$ & $X_9^{++}$ &
\multirow{4}{*}{$\Gamma_{\C'}\bigl(f_1^{1,1}, \ldots, f_6^{1,1}\bigr)$} &
\multirow{4}{*}{$\Gamma_{\R'}\bigl(f_1^{1,1}, \ldots, f_6^{1,1}\bigr)$} &
\begin{tabular}[x]{@{}l@{}}
    $\phantom{\cup}\; \Gamma_{\R'}\bigl(f_1^{1,1}\bigr)$ \\
    $\cup\; \Gamma_{\R'^{>-2}}\bigl(f_5^{1,1}\bigr)$
\end{tabular}
\\ \cline{1-2}\cline{5-5}

$X_9^{--}$ & $X_9^{--}$ &&&
\begin{tabular}[x]{@{}l@{}}
    $\phantom{\cup}\; \Gamma_{\R'}\bigl(f_1^{1,1}\bigr)$ \\
    $\cup\; \Gamma_{\R'^{<+2}}\bigl(f_3^{1,1}\bigr)$
\end{tabular}
\\ \cline{1-2}\cline{5-5}

$X_9^{++}$ & $X_9^{--}$ &&&
$\Gamma_{\R^{<-2}}\bigl(f_4^{1,1}\bigr)$
\\ \cline{1-2}\cline{5-5}

$X_9^{--}$ & $X_9^{++}$ &&&
$\Gamma_{\R^{>+2}}\bigl(f_6^{1,1}\bigr)$
\\ \hline

$X_9^{+-}$ & $X_9^{+-}$ &
\multirow{4}{*}{$\Gamma_{\C''}\bigl(f_1^{i,i}, \ldots, f_6^{i,i}\bigr)$} &
\multirow{4}{*}{$\Gamma_{\R}\bigl(f_1^{1,1}, f_2^{1,1}\bigr)$} &
\multirow{4}{*}{$\Gamma_{\R}\bigl(f_1^{1,1}\bigr)$} \\ \cline{1-2}

$X_9^{-+}$ & $X_9^{-+}$ &&& \\ \cline{1-2}

$X_9^{+-}$ & $X_9^{-+}$ &&& \\ \cline{1-2}

$X_9^{-+}$ & $X_9^{+-}$ &&& \\ \hline

$X_9^{++}$ & $X_9^{+-}$ &
\multirow{4}{*}{$\Gamma_{\C'}\bigl(f_1^{1,i}, \ldots, f_6^{1,i}\bigr)$} &
\multirow{4}{*}{$\{(-6,0), (0,0), (6,0)\}$} &
\multirow{4}{*}{$\varnothing$} \\ \cline{1-2}

$X_9^{++}$ & $X_9^{-+}$ &&& \\ \cline{1-2}

$X_9^{--}$ & $X_9^{+-}$ &&& \\ \cline{1-2}

$X_9^{--}$ & $X_9^{-+}$ &&& \\ \hline

$X_9^{+-}$ & $X_9^{++}$ &
\multirow{4}{*}{$\Gamma_{\C''}\bigl(f_1^{i,1}, \ldots, f_6^{i,1}\bigr)$} &
\multirow{4}{*}{$\{(0,-6), (0,0), (0,6)\}$} &
\multirow{4}{*}{$\varnothing$} \\ \cline{1-2}

$X_9^{-+}$ & $X_9^{++}$ &&& \\ \cline{1-2}

$X_9^{+-}$ & $X_9^{--}$ &&& \\ \cline{1-2}

$X_9^{-+}$ & $X_9^{--}$ &&& \\ \hline
\end{tabular}
\end{table}

\end{theorem}

\begin{theorem}\label{thm:J10}
The structure of the equivalence classes of the $J_{10}$ singularities is as
shown in Table~\ref{tab:J10_equivalences} where for $j = 1, \ldots, 6$ and
$\rho, \sigma \in \{-1, +1\}$, the function $f_j^{\rho, \sigma}$ is defined as
follows:
\begin{align*}
f_1^{\rho, \sigma}(a) &:= +\sqrt{\rho \sigma} \cdot a \,, \\
f_2^{\rho, \sigma}(a) &:= -\sqrt{\rho \sigma} \cdot a \,, \displaybreak[0] \\
f_3^{\rho, \sigma}(a)
&:= + \sqrt{\frac{-\rho \sigma (a^2-\rho \cdot 4) (a^2-\rho \cdot 9)
    + a (a^2-\rho \cdot 3) \sqrt{a^2-\rho \cdot 4}}{2(a^2-\rho \cdot 4)}}\,, \\
f_4^{\rho, \sigma}(a)
&:= - \sqrt{\frac{-\rho \sigma (a^2-\rho \cdot 4) (a^2-\rho \cdot 9)
    + a (a^2-\rho \cdot 3) \sqrt{a^2-\rho \cdot 4}}{2(a^2-\rho \cdot 4)}}\,,
    \displaybreak[0] \\
f_5^{\rho, \sigma}(a)
&:= + \sqrt{\frac{-\rho \sigma (a^2-\rho \cdot 4) (a^2-\rho \cdot 9)
    - a (a^2-\rho \cdot 3) \sqrt{a^2-\rho \cdot 4}}{2(a^2-\rho \cdot 4)}}\,, \\
f_6^{\rho, \sigma}(a)
&:= - \sqrt{\frac{-\rho \sigma (a^2-\rho \cdot 4) (a^2-\rho \cdot 9)
    - a (a^2-\rho \cdot 3) \sqrt{a^2-\rho \cdot 4}}{2(a^2-\rho \cdot 4)}}\,.
\end{align*}

In each case, $\rho$ and $\sigma$ are given by
\begin{align*}
\rho &:=
\begin{cases}
    +1, &\text{if } T_1 = J_{10}^+ \,, \\
    -1, &\text{if } T_1 = J_{10}^- \,, \\
\end{cases}
&\sigma &:=
\begin{cases}
    +1, &\text{if } T_2 = J_{10}^+ \,, \\
    -1, &\text{if } T_2 = J_{10}^- \,. \\
\end{cases}
\end{align*}

Furthermore, we use the following notations:
{\setlength{\jot}{5pt}   % manually adjusted
\begin{align*}
\xi  &:= \smash{\textstyle\frac{3}{\sqrt{2}}}    \,, &
I_1  &:= \left] -\infty, -\xi \right[ \subset \R \,, \\
\C'  &:= \C \setminus \{ -2, 2 \}                \,, &
I_2  &:= \left] -\xi, -2 \right[ \subset \R      \,, \\
\R'  &:= \R \setminus \{ -2, 2 \}                \,, &
I_3  &:= \left] +2, +\xi \right[ \subset \R      \,, \\
\C'' &:= \C \setminus \{ -2i, 2i \}              \,, &
I_4  &:= \left] +\xi, +\infty \right[ \subset \R \,.
\end{align*}}

\begin{table}[!htbp]
\centering
\caption{$P_1$, $P_2$ and $P_3$ for the $J_{10}$ singularities}
\label{tab:J10_equivalences}
\begin{tabular}{|c|c||c|c|c|}
\hline

$T_1$ & $T_2$ & $P_1(T_1, T_2)$ & $P_2(T_1, T_2)$ & $P_3(T_1, T_2)$ \\
\hline\hline

$J_{10}^+$ & $J_{10}^+$ &
$\Gamma_{\C'}\left(f_1^{\rho,\sigma}, \ldots, f_6^{\rho,\sigma}\right)$ &
\begin{tabular}[x]{@{}l@{}}
    $\phantom{\cup}\; \Gamma_{\R'}\left(f_1^{\rho,\sigma},
        f_2^{\rho,\sigma}\right)$ \\
    $\cup\; \Gamma_{\R^{>+2}}\left(f_3^{\rho,\sigma},
        f_4^{\rho,\sigma}\right)$ \\
    $\cup\; \Gamma_{\R^{<-2}}\left(f_5^{\rho,\sigma},
        f_6^{\rho,\sigma}\right)$ \\
    $\cup \, \{(0, -\xi), (0, +\xi)\}$ \\
    $\cup \, \{(-\xi, 0), (+\xi, 0)\}$ \\
\end{tabular} &
\begin{tabular}[x]{@{}l@{}}
    $\phantom{\cup}\; \Gamma_{\R'}\left(f_1^{\rho,\sigma}\right)$ \\
    $\cup\; \Gamma_{\R^{>+2}}\left(f_4^{\rho,\sigma}\right)$ \\
    $\cup\; \Gamma_{\R^{<-2}}\left(f_5^{\rho,\sigma}\right)$ \\
\end{tabular} \\
\hline

$J_{10}^-$ & $J_{10}^-$ &
$\Gamma_{\C''}\left(f_1^{\rho,\sigma}, \ldots, f_6^{\rho,\sigma}\right)$ &
$\Gamma_{\R}\left(f_1^{\rho,\sigma}, f_2^{\rho,\sigma}\right)$ &
$\Gamma_{\R}\left(f_1^{\rho,\sigma}\right)$ \\
\hline

$J_{10}^+$ & $J_{10}^-$ &
$\Gamma_{\C'}\left(f_1^{\rho,\sigma}, \ldots, f_6^{\rho,\sigma}\right)$ &
\begin{tabular}[x]{@{}l@{}}
    $\phantom{\cup}\; \{(0, 0)\}$ \\
    $\cup\; \Gamma_{\R^{>+2}}\left(f_3^{\rho,\sigma},
        f_4^{\rho,\sigma}\right)$ \\
    $\cup\; \Gamma_{\R^{<-2}}\left(f_5^{\rho,\sigma},
        f_6^{\rho,\sigma}\right)$ \\
\end{tabular} &
\begin{tabular}[x]{@{}l@{}}
    $\phantom{\cup}\; \Gamma_{I_4} \left(f_3^{\rho,\sigma}\right)$
    $\cup\; \Gamma_{I_3} \left(f_4^{\rho,\sigma}\right)$ \\
    $\cup\; \Gamma_{I_2} \left(f_5^{\rho,\sigma}\right)$
    $\cup\; \Gamma_{I_1} \left(f_6^{\rho,\sigma}\right)$ \\
    $\cup \, \{(-\xi, 0), (+\xi, 0)\}$ \\
\end{tabular} \\
\hline

$J_{10}^-$ & $J_{10}^+$ &
$\Gamma_{\C''}\left(f_1^{\rho,\sigma}, \ldots, f_6^{\rho,\sigma}\right)$ &
\begin{tabular}[x]{@{}l@{}}
    $\phantom{\cup}\; \{(0, 0)\}$ \\
    $\cup\; \Gamma_{\R}\left(f_3^{\rho,\sigma}, \ldots,
        f_6^{\rho,\sigma}\right)$ \\
\end{tabular} &
$\Gamma_{\R}\left(f_3^{\rho,\sigma}, f_6^{\rho,\sigma}\right)$ \\
\hline

\end{tabular}
\end{table}

\end{theorem}

\begin{remark}
In Theorem~\ref{thm:J10}, the definitions of
$f_1^{\rho,\sigma}, \ldots, f_6^{\rho,\sigma}$ involve square roots of possibly
complex values. These square roots are defined as follows: For any complex
number $z = re^{i\phi} \in \C$ with $r, \phi \in \R$, $r > 0$, and
$0 \leq \phi < 2 \pi$, we set
\[
\sqrt{z} := \sqrt{r}e^{i\frac{\phi}{2}} \,.
\]
In particular, $\Imag(\sqrt{z}) > 0$ for all $z \in \C \setminus \R^{>0}$ and
$\sqrt{z} \geq 0$ for all $z \in \R^{>0}$.
\end{remark}

\begin{theorem}\label{thm:J10+k}
The structure of the equivalence classes of the $J_{10+k}$
singularities is as shown in Table~\ref{tab:J10+k_equivalences} where in each
case, $l$ and $s$ are given by
\begin{align*}
l &:= \frac{6}{\gcd(6,k)}, \text{ and} \\
s &:=
\begin{cases}
  +1, &\text{if } k \equiv 2 \pmod{4}, \\
  -1, &\text{else.}
\end{cases}
\end{align*}

\begin{table}[!tbp]
\centering
\caption{$P_1$, $P_2$ and $P_3$ for the $J_{10+k}$ singularities}
\label{tab:J10+k_equivalences}
\begin{tabular}{|c|c||c|c|c|}
\hline

$T_1$ & $T_2$ & $P_1(T_1, T_2)$ & $P_2(T_1, T_2)$ & $P_3(T_1, T_2)$ \\
\hline\hline

$J_{10+k}^+$ & $J_{10+k}^+$ &
\multirow{2}{*}{$C(X^l-1)$} &
\multirow{2}{*}{$R(X^l-1)$} &
\multirow{2}{*}{$R(X^l-1)$} \\
\cline{1-2}

$J_{10+k}^-$ & $J_{10+k}^-$ &&& \\
\hline

$J_{10+k}^+$ & $J_{10+k}^-$ &
\multirow{2}{*}{$C(X^l-s)$} &
\multirow{2}{*}{$R(X^l-s)$} &
\multirow{2}{*}{$\varnothing$} \\
\cline{1-2}

$J_{10+k}^-$ & $J_{10+k}^+$ &&& \\
\hline

\end{tabular}
\end{table}

\end{theorem}

\begin{theorem}\label{thm:X9+k}
The structure of the equivalence classes of the $X_{9+k}$
singularities is as shown in Table~\ref{tab:X9+k_equivalences} where in each
case, $l$ and $s$ are given by
\begin{align*}
l &:= \frac{4}{\gcd(4,k)}, \text{ and} \\
s &:=
\begin{cases}
  +1, &\text{if } k \equiv 4 \pmod{8}, \\
  -1, &\text{else.}
\end{cases}
\end{align*}

\begin{table}[!htbp]
\centering
\caption{$P_1$, $P_2$ and $P_3$ for the $X_{9+k}$ singularities}
\label{tab:X9+k_equivalences}
\begin{tabular}{|c|c||c|c|c|}
\hline

$T_1$ & $T_2$ & $P_1(T_1, T_2)$ & $P_2(T_1, T_2)$ & $P_3(T_1, T_2)$ \\
\hline\hline

$X_{9+k}^{++}$ & $X_{9+k}^{++}$ &
\multirow{4}{*}{$C(X^l-1)$} &
\multirow{4}{*}{$R(X^l-1)$} &
\multirow{4}{*}{$R(X^{k+1}-1)$} \\
\cline{1-2}

$X_{9+k}^{+-}$ & $X_{9+k}^{+-}$ &&& \\ \cline{1-2}
$X_{9+k}^{-+}$ & $X_{9+k}^{-+}$ &&& \\ \cline{1-2}
$X_{9+k}^{--}$ & $X_{9+k}^{--}$ &&& \\ \hline

$X_{9+k}^{++}$ & $X_{9+k}^{+-}$ &
\multirow{4}{*}{$C(X^l-1)$} &
\multirow{4}{*}{$R(X^l-1)$} &
\multirow{4}{*}{$\varnothing$} \\
\cline{1-2}

$X_{9+k}^{+-}$ & $X_{9+k}^{++}$ &&& \\ \cline{1-2}
$X_{9+k}^{-+}$ & $X_{9+k}^{--}$ &&& \\ \cline{1-2}
$X_{9+k}^{--}$ & $X_{9+k}^{-+}$ &&& \\ \hline

$X_{9+k}^{++}$ & $X_{9+k}^{-+}$ &
\multirow{4}{*}{$C(X^l-s)$} &
\multirow{4}{*}{$R(X^l-s)$} &
\multirow{4}{*}{$\varnothing$} \\
\cline{1-2}

$X_{9+k}^{+-}$ & $X_{9+k}^{--}$ &&& \\ \cline{1-2}
$X_{9+k}^{-+}$ & $X_{9+k}^{++}$ &&& \\ \cline{1-2}
$X_{9+k}^{--}$ & $X_{9+k}^{+-}$ &&& \\ \hline

$X_{9+k}^{++}$ & $X_{9+k}^{--}$ &
\multirow{4}{*}{$C(X^l-s)$} &
\multirow{4}{*}{$R(X^l-s)$} &
\multirow{4}{*}{$\varnothing$} \\
\cline{1-2}

$X_{9+k}^{+-}$ & $X_{9+k}^{-+}$ &&& \\ \cline{1-2}
$X_{9+k}^{-+}$ & $X_{9+k}^{+-}$ &&& \\ \cline{1-2}
$X_{9+k}^{--}$ & $X_{9+k}^{++}$ &&& \\ \hline

\end{tabular}
\end{table}

\end{theorem}

\begin{theorem}\label{thm:Yrs}
The structure of the equivalence classes of the $Y_{r,s}$ singularities is as
shown in Table~\ref{tab:Yrs_equivalences} where in each case, $l$, $s_1$ and
$s_2$ are given by
\begin{align*}
l &:= \frac{r}{\gcd(r,s)} \cdot \gcd(2, r+1, s+1) \,, \displaybreak[0] \\
s_1 &:=
\begin{cases}
  +1, &\text{if } r \equiv 0 \pmod{4} \text{ or } s \equiv 0 \pmod{4}, \\
  -1, &\text{else,}
\end{cases} \displaybreak[0] \\
s_2 &:=
\begin{cases}
  +1, &\text{if } r \not\equiv 0 \pmod{2}
      \text{ or } \frac{s}{\gcd(r,s)} \equiv 0 \pmod{2}, \\
  -1, &\text{else.}
\end{cases}
\end{align*}

In the special case where $r = s$, additional equivalences occur. They are
listed in Table~\ref{tab:Yrr_equivalences}.

\begin{table}[!htbp]
\centering
\caption{$P_1$, $P_2$ and $P_3$ for the $Y_{r,s}$ singularities}
\label{tab:Yrs_equivalences}
\begin{tabular}{|c|c||c|c|c|}
\hline

$T_1$ & $T_2$ & $P_1(T_1, T_2)$ & $P_2(T_1, T_2)$ & $P_3(T_1, T_2)$ \\
\hline\hline

$Y_{r,s}^{++}$ & $Y_{r,s}^{++}$ &
\multirow{4}{*}{$C(X^l-1)$} &
\multirow{4}{*}{$R(X^l-1)$} &
\multirow{4}{*}{$R(X^{s+1}-1)$}
\\ \cline{1-2}

$Y_{r,s}^{-+}$ & $Y_{r,s}^{-+}$ &&&
\\ \cline{1-2}

$Y_{r,s}^{+-}$ & $Y_{r,s}^{+-}$ &&&
\\ \cline{1-2}

$Y_{r,s}^{--}$ & $Y_{r,s}^{--}$ &&&
\\ \hline

$Y_{r,s}^{++}$ & $Y_{r,s}^{-+}$ &
\multirow{4}{*}{$C(X^l-s_1)$} &
\multirow{4}{*}{$R(X^l-s_1)$} &
\multirow{4}{*}{$\varnothing$}
\\ \cline{1-2}

$Y_{r,s}^{-+}$ & $Y_{r,s}^{++}$ &&&
\\ \cline{1-2}

$Y_{r,s}^{+-}$ & $Y_{r,s}^{--}$ &&&
\\ \cline{1-2}

$Y_{r,s}^{--}$ & $Y_{r,s}^{+-}$ &&&
\\ \hline

$Y_{r,s}^{++}$ & $Y_{r,s}^{+-}$ &
\multirow{4}{*}{$C(X^l-s_2)$} &
\multirow{4}{*}{$R(X^l-s_2)$} &
\multirow{4}{*}{$\begin{cases}
  R(X^l-s_2),  &\text{if } r \not\equiv 0 \pmod{2} \\
  \varnothing, &\text{if } r \equiv 0 \pmod{2}
\end{cases}$}
\\ \cline{1-2}

$Y_{r,s}^{-+}$ & $Y_{r,s}^{--}$ &&&
\\ \cline{1-2}

$Y_{r,s}^{+-}$ & $Y_{r,s}^{++}$ &&&
\\ \cline{1-2}

$Y_{r,s}^{--}$ & $Y_{r,s}^{-+}$ &&&
\\ \hline

$Y_{r,s}^{++}$ & $Y_{r,s}^{--}$ &
\multirow{4}{*}{$C(X^l-s_1 s_2)$} &
\multirow{4}{*}{$R(X^l-s_1 s_2)$} &
\multirow{4}{*}{$\varnothing$}
\\ \cline{1-2}

$Y_{r,s}^{-+}$ & $Y_{r,s}^{+-}$ &&&
\\ \cline{1-2}

$Y_{r,s}^{+-}$ & $Y_{r,s}^{-+}$ &&&
\\ \cline{1-2}

$Y_{r,s}^{--}$ & $Y_{r,s}^{++}$ &&&
\\ \hline

\end{tabular}
\end{table}

\begin{table}[!htbp]
\centering
\caption{Additional equivalences for the $Y_{r,s}$ singularities in the
special case $r = s$}
\label{tab:Yrr_equivalences}
\begin{tabular}{|c|c||c|}
\hline

$T_1$ & $T_2$ & Additional elements of $P_3(T_1, T_2)$ \\
\hline\hline

$Y_{r,s}^{++}$ & $Y_{r,s}^{+-}$ &
\multirow{2}{*}{$\begin{cases}
  R(X+1),      &\!\text{if } r \equiv 0 \pmod{2} \text{ and } a < 0 \\
  \varnothing, &else
\end{cases}$}
\\ \cline{1-2}

$Y_{r,s}^{-+}$ & $Y_{r,s}^{--}$ &
\\ \hline

$Y_{r,s}^{+-}$ & $Y_{r,s}^{++}$ &
\multirow{2}{*}{$\begin{cases}
  R(X+1),      &\!\text{if } r \equiv 0 \pmod{2} \text{ and } a > 0 \\
  \varnothing, &else
\end{cases}$}
\\ \cline{1-2}

$Y_{r,s}^{--}$ & $Y_{r,s}^{-+}$ &
\\ \hline

\end{tabular}
\end{table}

\end{theorem}

\begin{remark}
Note that there are also equivalences between subtypes of $Y_{r,s}$ and
subtypes of $Y_{s,r}$ which can be obtained by just swapping the variables $x$
and $y$. For $r = s$ these are exactly the additional equivalences listed in
Table~\ref{tab:Yrr_equivalences}. But equivalences of this kind also occur for
$r \neq s$, e.g.\@ we have $R(X-1) \subset P_1(Y_{5,7}^{++}, Y_{7,5}^{++})$,
but we do not consider those cases in Theorem~\ref{thm:Yrs}.
\end{remark}

\begin{theorem}\label{thm:tYr}
The structure of the equivalence classes of the $\tY_r$ singularities is as
shown in Table~\ref{tab:tYr_equivalences} where in each case, $l$ and $s$ are
given by
\begin{align*}
l &:= 2\cdot{\gcd(2, r+1)}, \text{ and} \\
s &:=
\begin{cases}
  +1, &\text{if } r \equiv 0 \pmod{4}, \\
  -1, &\text{else.}
\end{cases}
\end{align*}

\begin{table}[!htbp]
\centering
\caption{$P_1$, $P_2$ and $P_3$ for the $\tY_r$ singularities}
\label{tab:tYr_equivalences}
\begin{tabular}{|c|c||c|c|c|}
\hline

$T_1$ & $T_2$ & $P_1(T_1, T_2)$ & $P_2(T_1, T_2)$ & $P_3(T_1, T_2)$ \\
\hline\hline

$\tY_r^+$ & $\tY_r^+$ &
\multirow{2}{*}{$C(X^l-1)$} &
\multirow{2}{*}{$R(X^2-1)$} &
\multirow{2}{*}{$\begin{cases}
  R(X^2-1), &\text{if } r \equiv 1 \pmod{2} \\
  R(X-1),   &\text{if } r \equiv 0 \pmod{2}
\end{cases}$} \\
\cline{1-2}

$\tY_r^-$ & $\tY_r^-$ &
&
&
\\
\hline

$\tY_r^+$ & $\tY_r^-$ &
\multirow{2}{*}{$C(X^l-s)$} &
\multirow{2}{*}{$R(X^2-s)$} &
\multirow{2}{*}{$\varnothing$} \\
\cline{1-2}

$\tY_r^-$ & $\tY_r^+$ &
&
&
\\
\hline

\end{tabular}
\end{table}

\end{theorem}

\begin{theorem}\label{thm:exceptional}
The structure of the equivalence classes of the exceptional unimodal
singularities is as shown in Table~\ref{tab:exceptional_equivalences}.

\begin{table}[!htbp]
\centering
\caption{$P_1$, $P_2$ and $P_3$ for the exceptional unimodal singularities}
\label{tab:exceptional_equivalences}
\begin{tabular}{|c|c||c|c|c|}
\hline
$T_1$ & $T_2$ & $P_1(T_1, T_2)$ & $P_2(T_1, T_2)$ & $P_3(T_1, T_2)$ \\
\hline\hline
$E_{12}$   & $E_{12}$   & $C_0(X^{21}-1)$ & $R_0(X-1)$    & $R_0(X-1)$ \\
\hline
$E_{13}$   & $E_{13}$   & $C_0(X^{15}-1)$ & $R_0(X-1)$    & $R_0(X-1)$ \\
\hline
$E_{14}^+$ & $E_{14}^+$ & $C_0(X^{12}-1)$ & $R_0(X^2-1)$  & $R_0(X-1)$ \\
\hline
$E_{14}^-$ & $E_{14}^-$ & $C_0(X^{12}-1)$ & $R_0(X^2-1)$  & $R_0(X-1)$ \\
\hline
$E_{14}^+$ & $E_{14}^-$ & $C_0(X^{12}+1)$ & $\varnothing$ & $\varnothing$ \\
\hline
$E_{14}^-$ & $E_{14}^+$ & $C_0(X^{12}+1)$ & $\varnothing$ & $\varnothing$ \\
\hline
$Z_{11}$   & $Z_{11}$   & $C_0(X^{15}-1)$ & $R_0(X-1)$    & $R_0(X-1)$ \\
\hline
$Z_{12}$   & $Z_{12}$   & $C_0(X^{11}-1)$ & $R_0(X-1)$    & $R_0(X-1)$ \\
\hline
$Z_{13}^+$ & $Z_{13}^+$ & $C_0(X^9-1)$    & $R_0(X-1)$    & $R_0(X-1)$ \\
\hline
$Z_{13}^-$ & $Z_{13}^-$ & $C_0(X^9-1)$    & $R_0(X-1)$    & $R_0(X-1)$ \\
\hline
$Z_{13}^+$ & $Z_{13}^-$ & $C_0(X^9+1)$    & $R_0(X+1)$    & $\varnothing$ \\
\hline
$Z_{13}^-$ & $Z_{13}^+$ & $C_0(X^9+1)$    & $R_0(X+1)$    & $\varnothing$ \\
\hline
$W_{12}^+$ & $W_{12}^+$ & $C_0(X^{10}-1)$ & $R_0(X^2-1)$  & $R_0(X-1)$ \\
\hline
$W_{12}^-$ & $W_{12}^-$ & $C_0(X^{10}-1)$ & $R_0(X^2-1)$  & $R_0(X-1)$ \\
\hline
$W_{12}^+$ & $W_{12}^-$ & $C_0(X^{10}+1)$ & $\varnothing$ & $\varnothing$ \\
\hline
$W_{12}^-$ & $W_{12}^+$ & $C_0(X^{10}+1)$ & $\varnothing$ & $\varnothing$ \\
\hline
$W_{13}^+$ & $W_{13}^+$ & $C_0(X^8-1)$    & $R_0(X^2-1)$  & $R_0(X-1)$ \\
\hline
$W_{13}^-$ & $W_{13}^-$ & $C_0(X^8-1)$    & $R_0(X^2-1)$  & $R_0(X-1)$ \\
\hline
$W_{13}^+$ & $W_{13}^-$ & $C_0(X^8+1)$    & $\varnothing$ & $\varnothing$ \\
\hline
$W_{13}^-$ & $W_{13}^+$ & $C_0(X^8+1)$    & $\varnothing$ & $\varnothing$ \\
\hline
\end{tabular}
\end{table}

\end{theorem}

\section{Interpretation of the Results}\label{sec:interpretation}

Looking more closely at Theorem~\ref{thm:sufficient_sets} and
Section~\ref{sec:results}, it turns out that the structure of the equivalence
classes is quite simple for some singularity types whereas it is very involved
for others. To describe this in more detail, let us first consider the
\emph{sufficient sets} given in Theorem~\ref{thm:sufficient_sets}:
\begin{itemize}
\item
It suffices to work with scalings of the form $\phi(x) = \alpha x$,
$\phi(y) = \beta y$ as coordinate transformations to figure out the structure
of the equivalence classes of the hyperbolic and exceptional unimodal
singularities. (To be precise, for $Y_{r,s}$ with $r = s$ we also have to take
into account transformations of the form $\phi(x) = \beta y$,
$\phi(y) = \alpha x$ where the variables are swapped, cf.\@
Table~\ref{tab:sufficient_sets}.)
\item
For the two parabolic types $X_9$ and $J_{10}$, scalings are not sufficient.
Instead, we have to consider more complicated transformations which involve
more terms.
\item
One can see from the proof of Theorem~\ref{thm:sufficient_sets} that these
differences reflect the different shapes of the normal forms: The normal forms
of the hyperbolic singularity types listed in Table~\ref{tab:sufficient_sets}
are weighted quasihomogeneous, those of the exceptional singularity types are
semi-quasihomogeneous. Both shapes turn out to be very restrictive w.r.t.\@
possible coordinate transformations. In contrast to this, the normal forms of
the parabolic types are quasihomogeneous and thus allow for more freedom in
this regard.
\item
The singularity type $\tY_r$ is an exception. It is complex equivalent to
$Y_{r,r}$, but it appears as a separate singularity type over $\R$. There is no
degree-bounded sufficient set for the normal form of this type (cf.\@
Remark~\ref{rem:sufficient_sets_for_Yr}), so the computational methods
described in Sections~\ref{ssec:computing_P1} and \ref{ssec:computing_P3} do
not work. Instead, we have to use other methods, cf.\@ Section~\ref{ssec:Yr}.
\end{itemize}

As a consequence of the differences w.r.t.\@ sufficient sets described above,
there are two \emph{general forms of equivalences} as presented in
Section~\ref{sec:results}:
\begin{itemize}
\item
For the hyperbolic and the exceptional singularities, the equivalences between
different subtypes can be described by constant factors. More precisely, if
$T_1$ and $T_2$ are subtypes of the same hyperbolic or exceptional main
singularity type, then there exists a finite set of constants
$r_1, \ldots, r_m \in \C$ such that the equivalences between the normal forms
of $T_1$ and $T_2$ are exactly those of the form
$\NF(T_1(a)) \sim \NF(T_2(r_i a))$ with $a \in \C$ or $a \in \R$ as
appropriate. Therefore we use the notations $C(p(X))$, $R(p(X))$ and
$C_0(p(X))$, $R_0(p(X))$ with $p(X) \in \C[X]$ (see Definition~\ref{def:C_R})
for the hyperbolic and the exceptional cases, respectively, cf.\@
Theorems~\ref{thm:J10+k} to \ref{thm:exceptional}.
\item
The equivalences which occur among subtypes of the two parabolic singularity
types $X_9$ and $J_{10}$ are much more involved and cannot be written down in
terms of constant factors. We describe them as joint graphs of certain
functions, cf.\@ Theorems~\ref{thm:X9} and \ref{thm:J10}.
\end{itemize}

The results presented in Section~\ref{sec:results} have \emph{consequences for
the algorithmic classification of the unimodal singularities of corank $2$ over
$\R$}. They are indeed intended to be the first step in this direction. Once
again, we can distinguish between different cases. Note that the following
remarks apply to the classification over $\R$ and therefore only deal with real
coordinate transformations and real values of the involved parameters:
\begin{itemize}
\item
In the exceptional cases, there are no equivalences between different subtypes
of the same main type and the value of the parameter is uniquely determined,
cf.\@ Theorem~\ref{thm:exceptional}.
\item
For the singularity types $J_{10+k}$, $X_{9+k}$, and $\tY_r$, there are no
equivalences between different subtypes of the same main type, but the
parameter can in some cases change its sign within the same subtype, cf.\@
Theorems~\ref{thm:J10+k}, \ref{thm:X9+k}, and \ref{thm:tYr}.
\item
For $Y_{r,s}$, there are equivalences even between different subtypes.
Therefore the question which real subtype a given singularity of main type
$Y_{r,s}$ belongs to is not always well-posed, e.g., a singularity can be both
of type $Y_{5,7}^{++}$ and of type $Y_{5,7}^{+-}$. However, the first of the
two signs is always uniquely determined. The parameter can change its sign, but
its absolute value is uniquely determined, cf.\@ Theorem~\ref{thm:Yrs}.
\item
The structures of the equivalence classes of the two parabolic cases $X_9$ and
$J_{10}$ are the most complicated among all the cases discussed here. There are
equivalences between different subtypes and the parameter may change in
non-trivial ways. For $X_9$, the possible values which a given parameter can be
transformed to can be expressed as rational functions of this parameter (cf.\@
Theorem~\ref{thm:X9}), whereas the corresponding functions for $J_{10}$ involve
radical expressions (cf.\@ Theorem~\ref{thm:J10}). Note that there are,
however, no equivalences between the subtypes $X_9^{++}$, $X_9^{--}$ on the one
hand and $X_9^{+-}$, $X_9^{-+}$ on the other hand, i.e.\@ the product of the
two signs which occur in the subtypes of $X_9$ is uniquely determined.
\item
It is a remarkable result that for any value of $a \in \R$, the normal form of
$J_{10}^-(a)$ is $\R$-equivalent to the normal form of $J_{10}^+(a')$ for some
$a' \in \R$ while the converse is not true, cf.\@ Theorem~\ref{thm:J10}. In
other words, the real subtype $J_{10}^-$ is redundant whereas $J_{10}^+$ is
not.
\end{itemize}

To sum up, the normal forms which are listed in the classifications of the
unimodal singularities over $\C$ and over $\R$ by \citet{AVG1985} cover the
whole space of unimodal singularities, but some of them are equivalent. There
are equivalences between normal forms for different values of the parameter and
also between different subtypes, to an extend that the subtype $J_{10}^-$ is
even redundant.

\section{Acknowledgements}

We would like to thank Claus Fieker, Gert-Martin Greuel, and Gerhard Pfister
for many fruitful discussions.


\clearpage

\begin{thebibliography}{99}

\bibitem[{Arnold(1974)}]{A1974}
Arnold, V.I., 1974.
Normal forms of functions in neighbourhoods of degenerate critical points.
Russ. Math. Surv. 29(2), 10-50.

\bibitem[{Arnold(1976)}]{A1976}
Arnold, V.I., 1976.
Local Normal Forms of Functions.
Invent. Math. 35, 87-109.

\bibitem[{Arnold et al.(1985)}]{AVG1985}
Arnold, V.I., Gusein-Zade, S.M., Varchenko, A.N., 1985.
Singularities of Differential Maps, Vol.~I.
Birkh\"auser, Boston.

\bibitem[{Bruce and Gaffney(1982)}]{BG1982}
Bruce, J.W., Gaffney, T.J., 1982.
Simple Singularities of Mappings $\C, 0 \rightarrow \C^2, 0$.
J.~London Math. Soc. 26(3), 465-474.

\bibitem[{Decker et al.(2015a)}]{DGPS}
Decker, W., Greuel, G.-M., Pfister, G., Sch\"onemann, H., 2015.
{\sc Singular} {4-0-2} -- {A} computer algebra system for polynomial
computations.
\url{http://www.singular.uni-kl.de}.

\bibitem[{Decker et al.(2015b)}]{primdec}
Decker, W., Laplagne, S., Pfister, G., Sch\"onemann, H., 2015.
{\tt primdec.lib}. {A} {\sc Singular} {4-0-2} library for primary decomposition
and radicals of ideals.

\bibitem[{Fr\"uhbis-Kr\"uger(1999)}]{FK1999}
Fr\"uhbis-Kr\"uger, A., 1999.
Classification of Simple Space Curve Singularities.
Commun. Algebra 27(8), 3993-4013.

\bibitem[{Fr\"uhbis-Kr\"uger and Neumer(2010)}]{FKN2010}
Fr\"uhbis-Kr\"uger, A., Neumer, A., 2010.
Simple Cohen-Macaulay Codimension 2 Singularities.
Commun. Algebra 38(2), 454-495.

\bibitem[{Gibson and Hobbs(1993)}]{GH1993}
Gibson, C.G., Hobbs, C.A., 1993.
Simple Singularities of Space Curves.
Proc. Camb. Philos. Soc. 113(2), 297-310.

\bibitem[{Greuel et al.(2007)}]{GLS2007}
Greuel, G.-M., Lossen, C., Shustin, E., 2007.
Introduction to Singularities and Deformations.
Springer, Berlin.

\bibitem[{Greuel and Pfister(2008)}]{GP2008}
Greuel, G.-M., Pfister, G., 2008.
A Singular Introduction to Commutative Algebra, second ed.
Springer, Berlin.

\bibitem[{Marais and Steenpa\ss(2015)}]{realclassify}
Marais, M.S., Steenpa\ss, A., 2015.
{\tt realclassify.lib}. {A} {\sc Singular} {4-0-2} library for classifying
isolated hypersurface singularities over the reals w.r.t.\@ right
equiva\-lence.

\bibitem[{Marais and Steenpa\ss(2015)}]{MS2015}
Marais, M.S., Steenpa\ss, A., 2015.
The Classification of Real Singularities Using \Singular{}. Part I: Splitting
Lemma and Simple Singularities.
J. Symb. Comput. 68, 61-71.

\bibitem[{Montes and Sch\"onemann(2015)}]{grobcov}
Montes, A., Sch\"onemann, H., 2015.
{\tt grobcov.lib}. {A} {\sc Singular} {4-0-2} library for Gr\"obner covers of
parametric ideals.

\bibitem[{Siersma(1974)}]{Siersma}
Siersma, D., 1974.
Classification and Deformation of Singularities.
Dissertation, University of Amsterdam.

\end{thebibliography}
\end{document}